\documentclass[12pt,a4paper]{article}
\usepackage{amssymb}
\usepackage{amsmath}
\usepackage{latexsym}
\usepackage{amsthm}

\def\E{\mathbb E}

\newtheorem{lem}{Lemma}
\newtheorem{thm}{Theorem}
\newtheorem{thmletter}{Theorem}

\newtheorem{cor}[thm]{Corollary}
\newtheorem{prop}[thm]{Proposition}
\theoremstyle{definition}
\newtheorem*{remark}{Remark}

\usepackage{xpatch}
\xpatchcmd{\proof}{\itshape}{\normalfont\proofnameformat}{}{}
\newcommand{\proofnameformat}{}

\begin{document}

\renewcommand{\proofnameformat}{\bfseries}

\begin{center}
{\large\textbf{Equidistribution of random walks on compact groups}}

\vspace{5mm}

\textbf{Bence Borda}

{\footnotesize Alfr\'ed R\'enyi Institute of Mathematics, Hungarian Academy of Sciences

1053 Budapest, Re\'altanoda u.\ 13--15, Hungary

Email: \texttt{borda.bence@renyi.mta.hu}}

\vspace{5mm}

{\footnotesize \textbf{Keywords:} empirical distribution, ergodic theorem, strong law of large numbers, law of the iterated logarithm, central limit theorem, Wiener process}

{\footnotesize \textbf{Mathematics Subject Classification (2010):} 60G50, 60B15}
\end{center}

\vspace{5mm}

\begin{abstract}
Let $X_1, X_2, \dots$ be independent, identically distributed random variables taking values from a compact metrizable group $G$. We prove that the random walk $S_k=X_1 X_2 \cdots X_k$, $k=1,2,\dots$ equidistributes in any given Borel subset of $G$ with probability $1$ if and only if $X_1$ is not supported on any proper closed subgroup of $G$, and $S_k$ has an absolutely continuous component for some $k \ge 1$. More generally, the sum $\sum_{k=1}^N f(S_k)$, where $f:G \to \mathbb{R}$ is Borel measurable, is shown to satisfy the strong law of large numbers and the law of the iterated logarithm. We also prove the central limit theorem with remainder term for the same sum, and construct an almost sure approximation of the process $\sum_{k \le t} f(S_k)$ by a Wiener process provided $S_k$ converges to the Haar measure in the total variation metric.
\end{abstract}

\section{Introduction}

Let $G$ be a compact Hausdorff group with normalized Haar measure $\mu$, and let $X_1, X_2, \dots$ be a sequence of independent, identically distributed (i.i.d.\ for short) $G$-valued random variables. The random walk $S_k=\prod_{j=1}^k X_j = X_1 X_2 \cdots X_k$ is a classical object in probability theory. Throughout we assume that the distribution of $X_1$ is a regular Borel probability measure $\nu$ on $G$; the distribution of $S_k$ is thus $\nu^{*k}$, the $k$-fold convolution of $\nu$. Generalizing results of L\'evy \cite{L} on the circle group $G=\mathbb{R} / \mathbb{Z}$ and Kawada and It\^o \cite{KI} on compact metrizable groups, it was Urbanik \cite{U} and Kloss \cite{K} who proved that if $\nu^{*k}$ is weakly convergent, then its weak limit is the normalized Haar measure of a closed subgroup of $G$. Stromberg \cite{ST} gave the following necessary and sufficient condition for the weak limit to be the Haar measure on $G$ itself. We shall say that $\nu$ is adapted if the support of $\nu$ is not contained in any proper closed subgroup of $G$, and that $\nu$ is strictly aperiodic if the support of $\nu$ is not contained in a coset of any proper closed normal subgroup of $G$.
\begin{thmletter}\label{weakconvergence} Let $G$ be a compact Hausdorff group, and let $\nu$ be a regular Borel probability measure on $G$. The following are equivalent.
\begin{enumerate}
\item[(i)] $\nu$ is adapted and strictly aperiodic.
\item[(ii)] $\nu^{*k} \to \mu$ weakly as $k \to \infty$.
\end{enumerate}
\end{thmletter}
\noindent A similar classical result gives a necessary and sufficient condition for convergence in the total variation metric $\| \cdot \|_{\mathrm{TV}}$. The Lebesgue decomposition of $\nu^{*k}$ with respect to the Haar measure $\mu$ will be written as $\nu^{*k}=(\nu^{*k})_{\mathrm{abs}}+(\nu^{*k})_{\mathrm{sing}}$, where $(\nu^{*k})_{\mathrm{abs}}$ is absolutely continuous and $(\nu^{*k})_{\mathrm{sing}}$ is singular with respect to $\mu$.
\begin{thmletter}\label{normconvergence} Let $G$ be a compact Hausdorff group, and let $\nu$ be a regular Borel probability measure on $G$. The following are equivalent.
\begin{enumerate}
\item[(i)] $\nu$ is adapted and strictly aperiodic, and $(\nu^{*k})_{\mathrm{abs}} \neq 0$ for some $k \ge 1$.
\item[(ii)] $\| \nu^{*k} -\mu \|_{\mathrm{TV}} \to 0$ as $k \to \infty$.
\end{enumerate}
Moreover, if these equivalent conditions hold, then the convergence in (ii) is exponentially fast.
\end{thmletter}
\noindent Special cases of Theorem \ref{normconvergence} were proved by Bhattacharya \cite{B}. For the general case and the history of related results see \cite{AG} and \cite{RX}. We mention that if $G$ is connected, then the assumption that $(\nu^{*k})_{\mathrm{abs}} \neq 0$ for some $k \ge 1$ implies that $\nu$ is adapted and strictly aperiodic. This follows from the fact that in a connected, compact Hausdorff group any proper closed subgroup has Haar measure $0$.

Theorems \ref{weakconvergence} and \ref{normconvergence} concern the distribution of $S_k$ for a given $k \ge 1$. We can also view $S_k$, $k=1,2,\dots$ as a random sequence in $G$, and consider the empirical distribution of the terms $S_1, S_2, \dots, S_N$ for some $N \ge 1$. Under the technical assumption that $G$ is metrizable, Berger and Evans \cite[Corollary 3.1]{BE} proved the following.
\begin{thmletter}\label{uniformdistribution} Let $G$ be a compact metrizable group. Let $X_1, X_2, \dots$ be i.i.d.\ $G$-valued random variables with distribution $\nu$, and set $S_k=\prod_{j=1}^k X_j$. The following are equivalent.
\begin{enumerate}
\item[(i)] $\nu$ is adapted.
\item[(ii)] For any continuous function $f:G \to \mathbb{R}$
\begin{equation}\label{LLN}
\lim_{N \to \infty} \frac{1}{N} \sum_{k=1}^N f(S_k) = \int_G f \, \mathrm{d} \mu \quad \mathrm{a.s.}
\end{equation}
\end{enumerate}
\end{thmletter}
\noindent Note that a.s.\ (almost surely) means that the given relation holds with probability $1$. Since the Banach space of continuous, real-valued functions on $G$ (or indeed, on any compact metric space) is separable, condition (ii) in the previous theorem is equivalent to the property that with probability $1$, \eqref{LLN} holds for all continuous functions $f: G \to \mathbb{R}$ simultaneously. A (deterministic) sequence $a_k$, $k=1,2,\dots$ in $G$ is called uniformly distributed if $\lim_{N \to \infty} (1/N) \sum_{k=1}^N f(a_k) = \int_G f \, \mathrm{d} \mu$ for any continuous function $f:G \to \mathbb{R}$. Theorem \ref{uniformdistribution} thus states that the random sequence $S_k$, $k=1,2,\dots$ is uniformly distributed with probability $1$ if and only if $\nu$ is adapted. See \cite{BB}, \cite{BR} and \cite{SCH} for related results on the circle group $G=\mathbb{R} / \mathbb{Z}$, and \cite{BE} for the case of continuous time processes.

In this paper we consider $\sum_{k=1}^N f(S_k)$ for Borel measurable functions \mbox{$f: G \to \mathbb{R}$,} and we prove the following analogue of Theorem \ref{uniformdistribution}.
\begin{thm}\label{stronguniformityLIL} Let $G$ be a compact metrizable group. Let $X_1, X_2, \dots$ be i.i.d.\ $G$-valued random variables with distribution $\nu$, and set $S_k=\prod_{j=1}^k X_j$. The following are equivalent.
\begin{enumerate}
\item[(i)] $\nu$ is adapted, and $(\nu^{*k})_{\mathrm{abs}} \neq 0$ for some $k \ge 1$.
\item[(ii)] For any bounded, Borel measurable function $f:G \to \mathbb{R}$
\begin{equation}\label{strongLLN}
\lim_{N \to \infty} \frac{1}{N} \sum_{k=1}^N f(S_k) = \int_G f \, \mathrm{d} \mu \quad \mathrm{a.s.}
\end{equation}
\item[(iii)] For any bounded, Borel measurable function $f:G \to \mathbb{R}$
\[ \limsup_{N \to \infty} \frac{\left| \sum_{k=1}^N f(S_k) -N \int_G f \, \mathrm{d} \mu \right|}{\sqrt{N \log \log N}} < \infty \quad \mathrm{a.s.} \]
\end{enumerate}
\end{thm}
\noindent The implications (iii)$\Rightarrow$(ii)$\Rightarrow$(i) are straightforward. Condition (ii) is in fact equivalent to the assumption that \eqref{strongLLN} holds for the indicator function $f=I_B$ of any Borel set $B \subseteq G$; indeed, a bounded, Borel measurable function can be uniformly approximated by finite linear combinations of such indicator functions. The equivalence (i)$\Leftrightarrow$(ii) in Theorem \ref{stronguniformityLIL} thus states that the random sequence $S_k$, $k=1,2,\dots$ equidistributes in any given Borel set with probability $1$ if and only if $\nu$ is adapted, and $(\nu^{*k})_{\mathrm{abs}} \neq 0$ for some $k \ge 1$. Equidistribution of a random sequence in any given Borel set with probability $1$ is sometimes called the ``strong uniform distribution'' property. In contrast, (ordinary) uniform distribution means equidistribution in any Borel set $B \subseteq G$ such that $\mu (\partial B)=0$. Note that equidistribution in all Borel sets simultaneously is impossible; in particular, no deterministic sequence satisfies the strong uniform distribution property (unless $G$ is finite).

In Theorems \ref{uniformdistribution} and \ref{stronguniformityLIL} we did not assume that $\nu$ is strictly aperiodic, whereas in Theorems \ref{weakconvergence} and \ref{normconvergence} strict aperiodicity is required for the convergence of $\nu^{*k}$. In the proof of the implication (i)$\Rightarrow$(iii) in Theorem \ref{stronguniformityLIL} we will thus first assume that $\nu$ is strictly aperiodic. In case the support of $\nu$ is contained in a coset of a closed normal subgroup $H \lhd G$, we will see that the factor group $G/H$ is finite and cyclic, and we will argue by induction on the index $|G:H|$. Surprisingly, in Theorem \ref{stronguniformityLIL} the strong law of large numbers (condition (ii)) and the law of the iterated logarithm (condition (iii)) are equivalent. This is a consequence of the fact that whenever $\nu^{*k}$ converges to the Haar measure $\mu$ in the total variation metric, the convergence is necessarily exponentially fast. This fact does not have an analogue for weak convergence. We also prove the following central limit theorem under the technical assumption that $\nu$ is a central measure. Note that condition (ii) below expresses convergence in distribution to the standard normal distribution.
\begin{thm}\label{stronguniformityCLT} Let $G$ be a compact metrizable group. Let $X_1, X_2, \dots$ be i.i.d.\ $G$-valued random variables with distribution $\nu$, and set $S_k=\prod_{j=1}^k X_j$. Assume that $\nu$ is central. The following are equivalent.
\begin{enumerate}
\item[(i)] $\nu$ is adapted and strictly aperiodic, and $(\nu^{*k})_{\mathrm{abs}} \neq 0$ for some $k \ge 1$.
\item[(ii)] For any bounded, Borel measurable function $f:G \to \mathbb{R}$ such that $\int_G f \, \mathrm{d} \mu =0$ and $f$ is not $\mu$-a.e.\ zero, we have
\[ \frac{\sum_{k=1}^N f(S_k)}{\sqrt{C(f,\nu ) N}} \overset{d}{\longrightarrow} \mathcal{N}(0,1) \]
with some constant $C(f,\nu)>0$ depending only on $f$ and $\nu$.
\end{enumerate}
\end{thm}
\noindent Recall that a compact Hausdorff topological space is metrizable if and only if it is second countable. We mention that in the proofs of Theorems \ref{stronguniformityLIL} and \ref{stronguniformityCLT} the choice of the metric on $G$ is irrelevant; we only use the second countability of $G$. Whether Theorems \ref{stronguniformityLIL} and \ref{stronguniformityCLT} are true for compact Hausdorff groups remains open.

\section{Results}

\subsection{Preliminaries}

For the rest of the paper we assume that $G$ is a compact metrizable group. Let $\mu$ denote the Haar measure on $G$ normalized so that $\mu (G)=1$. We will write $L^p(G)=L^p(G,\mu)$ for the Lebesgue space of real-valued, Borel measurable functions with respect to $\mu$, and $\| f \|_p= \| f \|_{L^p(G,\mu)}$. In addition, $\| \cdot \|_p$ will also denote the $L^p$-norm of (real-valued) random variables. Recall that $\mu$ is both left and right invariant; that is, for any Borel set $B \subseteq G$ and any $y \in G$ we have $\mu (yB)=\mu (By)=\mu (B)$. Therefore for any $f \in L^1(G)$ we have
\[ \int_G f(x) \, \mathrm{d}\mu (x)= \int_G f(xy) \, \mathrm{d}\mu (x) = \int_G f(yx) \, \mathrm{d}\mu (x). \]
The total variation of a finite, signed Borel measure $\vartheta$ on $G$ is defined as
\[ \| \vartheta \|_{\mathrm{TV}} = \sup \left\{ \left| \int_G f \, \mathrm{d} \vartheta \right| : f: G \to \mathbb{R} \textrm{ Borel measurable, } \sup_G |f| \le 1 \right\} . \]
Given two finite, signed Borel measures $\vartheta_1$ and $\vartheta_2$ on $G$, their convolution \mbox{$\vartheta_1 * \vartheta_2$} is the unique finite, signed Borel measure such that for any bounded, Borel measurable function $f: G \to \mathbb{R}$
\[ \int_G f \, \mathrm{d} (\vartheta_1 * \vartheta_2 ) = \int_G \int_G f (xy) \, \mathrm{d} \vartheta_1 (x) \mathrm{d} \vartheta_2 (y). \]
It is easy to see that $\| \vartheta_1 * \vartheta_2 \|_{\mathrm{TV}} \le \| \vartheta_1 \|_{\mathrm{TV}} \cdot \| \vartheta_2 \|_{\mathrm{TV}}$. A finite, signed Borel measure $\vartheta$ on $G$ is called central if $\vartheta (y^{-1}By)=\vartheta (B)$ for all Borel sets $B \subseteq G$ and all $y \in G$. Similarly, a Borel measurable function $f:G \to \mathbb{R}$ is called a class function if $f(y^{-1}xy)=f(x)$ for all $x,y \in G$. Note that for any Borel probability measure $\nu$ on $G$ we have $\nu * \mu = \mu * \nu = \mu$. If $\nu_1$ and $\nu_2$ are Borel probability measures on $G$, then $\| \nu_1 -\nu_2 \|_{\mathrm{TV}}=2 \sup_B |\nu_1 (B) - \nu_2 (B)|$, where the supremum is over all Borel sets $B \subseteq G$. The support of a Borel probability measure $\nu$ on $G$, denoted by $\mathrm{supp}\, \nu$, is the smallest closed set $F \subseteq G$ such that $\nu (F)=1$.

\begin{remark} Every finite Borel measure on $G$ (or indeed, on any Polish space) is regular. Therefore in the definitions of total variation and convolution we could have equivalently used continuous functions $f: G \to \mathbb{R}$ instead of bounded, Borel measurable ones. The existence and uniqueness of the convolution thus follows from the Riesz representation theorem.
\end{remark}

A $G$-valued random variable is a Borel measurable map $X$ from a probability space to $G$. Let $\nu_X$ denote the distribution of $X$; that is, $\nu_X (B)=\Pr (X \in B)$ for all Borel sets $B \subseteq G$. The variable $X$ is called uniformly distributed if $\nu_X = \mu$. If $X$ and $Y$ are independent $G$-valued random variables, then $\nu_{XY}=\nu_X * \nu_Y$. We shall write $X \overset{d}{=} Y$ if the (real-valued or $G$-valued) random variables $X$ and $Y$ have the same distribution.

Let $\hat{G}$ denote the unitary dual of $G$; that is, a complete set of pairwise unitarily inequivalent, irreducible unitary representations of $G$. Recall that every such representation is finite dimensional, and let $d_{\pi}$ denote the dimension of $\pi \in \hat{G}$. Thus $\pi(x)$ is a $d_{\pi} \times d_{\pi}$ unitary matrix with complex entries for any given $x \in G$. Let $\pi_0 \in \hat{G}$, $\pi_0=1$ denote the trivial representation. Given $f \in L^1(G)$ and $\pi \in \hat{G}$, let $\hat{f}(\pi)=\int_G f(x) \pi(x)^* \, \mathrm{d}\mu (x)$ denote the Fourier coefficients of $f$. Here $\pi(x)^*$ denotes the conjugate transpose of $\pi(x)$, and the integral is taken entrywise. The Fourier coefficients of a finite, signed Borel measure $\vartheta$ on $G$ are defined similarly as $\hat{\vartheta}(\pi )=\int_G \pi(x)^* \, \mathrm{d} \vartheta (x)$, $\pi \in \hat{G}$. The Parseval formula states that for any $f,g \in L^2(G)$ we have
\[ \int_G f g \, \mathrm{d} \mu = \sum_{\pi \in \hat{G}} d_{\pi} \mathrm{tr} \left( \hat{f}(\pi ) \hat{g} (\pi )^* \right) , \]
where $\mathrm{tr}$ denotes trace. Given $\pi \in \hat{G}$ the complex conjugate $\bar{\pi}$ is also an irreducible unitary representation of $G$, called the contragradient of $\pi$. The contragradient $\bar{\pi}$ may or may not be unitarily equivalent to $\pi$. For the theory of Fourier analysis on compact groups we refer the reader to \cite{FO}.

The notation $a_n \ll b_n$ and $a_n=O(b_n)$ mean that there exists an (implied) constant $K \ge 0$ such that $|a_n| \le K b_n$ for all $n \ge 1$. We write $a_n=\Theta (b_n)$ if $a_n \ll b_n \ll a_n$. We will use subscripts to denote dependence of the implied constant on certain parameters; thus e.g. $a_n \ll_{f,\nu} b_n$ and $a_n=O_{f,\nu}(b_n)$ mean that the implied constant may depend on $f$ and $\nu$. We emphasize that in the statement of all theorems, propositions and lemmas implied constants in the notation $\ll$ and $O$ are universal; in particular, they do not even depend on the group $G$.

\subsection{The main theorems}

Let $G$ be a compact metrizable group with normalized Haar measure $\mu$, let $X_1, X_2, \ldots$ be i.i.d.\ $G$-valued random variables with distribution $\nu$, and set $S_k = \prod_{j=1}^k X_j$. Let $\Delta_k= \| \nu^{*k}-\mu \|_{\mathrm{TV}}=2 \sup_B |\Pr (S_k \in B) - \mu (B)|$ denote the total variation distance of the distribution of $S_k$ from the uniform distribution. Note that
\[ \| \nu^{*{(k+1)}}-\mu \|_{\mathrm{TV}} = \| (\nu^{*k}-\mu)*\nu \|_{\mathrm{TV}} \le \| \nu^{*k}-\mu \|_{\mathrm{TV}} \cdot \| \nu \|_{\mathrm{TV}} = \| \nu^{*k}-\mu \|_{\mathrm{TV}} , \]
hence $\Delta_{k+1} \le \Delta_k$.

The precise rate of convergence in the total variation metric was found by Anoussis and Gatzouras \cite[Theorem 4.1]{AG}. Let
\[ q = \max \left\{ \sup_{\pi \in \hat{G} \backslash \{ \pi_0 \}} \rho (\hat{\nu} (\pi )), \,\,\, \inf_{k \ge 1} \| (\nu^{*k})_{\mathrm{sing}} \|_{\mathrm{TV}}^{1/k} \right\} , \]
where $\rho(\hat{\nu} (\pi ))$ denotes the spectral radius of the matrix $\hat{\nu} (\pi )$. Then $\lim_{k \to \infty} \Delta_k^{1/k} = \inf_{k \ge 1} \Delta_k^{1/k}=q$; moreover, $q<1$ if and only if $\nu$ is adapted and strictly aperiodic, and $(\nu^{*k})_{\mathrm{abs}}\neq 0$ for some $k \ge 1$. Thus, as already stated in Theorem \ref{normconvergence}, whenever $\| \nu^{*k}-\mu \|_{\mathrm{TV}} \to 0$, the convergence is necessarily exponentially fast; more precisely, we have $q^k \le \Delta_k$ for every $k \ge 1$, and $\Delta_k \le (q+\varepsilon)^k$ for every $k \ge k_0(\nu, \varepsilon)$. Let $\Delta=1+2\sum_{k=1}^{\infty} \Delta_k$, and observe $1/(1-q) \le \Delta$.

We now give a more quantitative form of Theorem \ref{stronguniformityLIL}. For any $m \ge 1$ and $\varepsilon >0$ let $\varphi_{m,\varepsilon} (N) = N \left( \prod_{i=1}^{m-1} \log_i N \right) \left( \log_m N \right)^{1+\varepsilon}$, where $\log_1 N=\log N$ and $\log_i N = \log \log_{i-1} N$ denotes the $i$-fold iterated logarithm.
\begin{thm}\label{maintheoremLIL} Suppose that $\nu$ is adapted, and that $(\nu^{*k})_{\mathrm{abs}} \neq 0$ for some $k \ge 1$. Let $f: G \to \mathbb{R}$ be Borel measurable such that $\int_G f \, \mathrm{d} \mu =0$.
\begin{enumerate}
\item[(i)] If $\sup_{c \in G} \E |f(cX_1)|^p <\infty$ for some $1 \le p \le 2$, then for any $m \ge 1$ and $\varepsilon >0$
\begin{equation}\label{mainSLLN}
\lim_{N \to \infty} \frac{\sum_{k=1}^N f(S_k)}{\varphi_{m,\varepsilon}(N)^{1/p}} = 0 \quad \mathrm{a.s.}
\end{equation}
\item[(ii)] If $\sup_{c \in G} \E |f(cX_1)|^p <\infty$ for some $p>2$, then
\begin{equation}\label{mainLIL}
\limsup_{N \to \infty} \frac{\left| \sum_{k=1}^N f(S_k) \right|}{\sqrt{N \log \log N}} < \infty \quad \mathrm{a.s.}
\end{equation}
\end{enumerate}
\end{thm}

\begin{remark} Note that $\sup_{c \in G} \E |f(cX_1)|^p <\infty$ clearly implies $f \in L^p(G)$. To mention a sufficient condition, suppose that $\nu$ is absolutely continuous with density $\frac{\mathrm{d}\nu}{\mathrm{d}\mu}$. If there exists a H\"older conjugate pair $r,s \in [1,\infty ]$, $1/r+1/s=1$ such that $f \in L^{pr}(G)$ and $\frac{\mathrm{d}\nu}{\mathrm{d}\mu} \in L^s(G)$, then $\sup_{c \in G} \E |f(cX_1)|^p \le \| f \|_{pr}^p \| \frac{\mathrm{d}\nu}{\mathrm{d}\mu} \|_s < \infty$.
\end{remark}

Under the extra condition that $\nu$ is strictly aperiodic, we will approximate $\sum_{k=1}^N f(S_k)$ by a sum of independent random variables. Following Strassen \cite{STR}, we can even construct an almost sure approximation by a Wiener process. To state the precise form of this result, let us introduce the following technical definition. Given a function $E(t)$ positive on $(t_0, \infty)$ for some $t_0$, we shall say that two stochastic processes $Y(t)$ and $Z(t)$ in the Skorokhod space $D[0,\infty )$, possibly defined on different probability spaces, are $o(E(t))$-equivalent if there exist finitely many processes $Y_1(t), Y_2(t), \dots, Y_m(t)$ in $D[0,\infty )$ such that $Y_1 (t)=Y(t)$, $Y_m(t)=Z(t)$, and for all $1 \le i \le m-1$ one of the following hold:
\begin{enumerate}
\item[(i)] The processes $Y_i (t)$ and $Y_{i+1}(t)$, possibly defined on different probability spaces, have the same distribution.
\item[(ii)] The processes $Y_i(t)$ and $Y_{i+1}(t)$ are defined on the same probability space, and $\lim_{t \to \infty} (Y_i(t)-Y_{i+1}(t))/E(t) =0$ a.s.
\end{enumerate}
Roughly speaking (ignoring the different probability spaces), $Y(t)$ and $Z(t)$ being $o(E(t))$-equivalent thus means $Y(t)=Z(t)+o(E(t))$ a.s. Given $f \in L^2(G)$ with $\int_G f \, \mathrm{d} \mu =0$, let
\begin{equation}\label{Cf}
C(f, \nu )= \mathbb{E}f(U)^2+2\sum_{k=1}^{\infty} \mathbb{E} f(U) f(US_k),
\end{equation}
where $U$ is a uniformly distributed $G$-valued random variable, independent of $X_1, X_2, \ldots$. As we shall see, the series in \eqref{Cf} is absolutely convergent, and \mbox{$C(f, \nu) \ge 0$.}
\begin{thm}\label{maintheoremWIENER} Suppose that $\nu$ is adapted and strictly aperiodic, and that $(\nu^{*k})_{\mathrm{abs}} \neq 0$ for some $k \ge 1$. Let $f: G \to \mathbb{R}$ be Borel measurable such that $\int_G f \, \mathrm{d} \mu =0$. If $\sup_{c \in G} \E |f(cX_1)|^{2+\delta} <\infty$ for some $0<\delta <2$, then the processes $\sum_{k \le t} f(S_k)$ and $\sqrt{C(f,\nu )}W(t)$ are $o(t^{1/2-\delta /20})$-equivalent, where $W(t)$ is a standard Wiener process.
\end{thm}
The almost sure approximation by a Wiener process yields even more precise asymptotics than those in Theorem \ref{maintheoremLIL}; for instance, it shows that for strictly aperiodic $\nu$ the value of the limsup in \eqref{mainLIL} is $\sqrt{2 C(f, \nu )}$. The almost sure asymptotics as well as the limit distribution of continuous functionals of the process $\sum_{k \le t} f(S_k)$ also follow. Instead of the random step functions $\sum_{k \le t} f(S_k)$, we could have also used the piecewise linear functions $\sum_{k \le \lfloor t \rfloor} f(S_k) + \left( t-\lfloor t \rfloor \right) f(S_{\lfloor t \rfloor +1})$. In that case the $o(t^{1/2-\delta /20})$-equivalence holds in the space $C[0,\infty)$ as well.

In Theorem \ref{maintheoremWIENER} in general we only know $C(f, \nu) \ge 0$; in the case $C(f, \nu )=0$ the result simply states that $\sum_{k \le t} f(S_k)=o(t^{1/2-\delta /20})$ a.s. The natural question whether $C(f,\nu )>0$ is surprisingly delicate. We shall prove a necessary and sufficient condition in terms of irreducible unitary representations of $G$, see Proposition \ref{Cfalternativeprop} below. As this condition is rather cumbersome to use, we also give simpler criteria to ensure $C(f,\nu )>0$. In particular, we will show that under mild technical assumptions (e.g.\ $f$ is a class function or $\nu$ is a central measure) we have $C(f,\nu )=0$ if and only if $f=0$ $\mu$-a.e. We work out the details in Section \ref{section3}.

It clearly follows from Theorem \ref{maintheoremWIENER} that $N^{-1/2} \sum_{k=1}^N f(S_k)$ has a (possibly degenerate) Gaussian limit distribution. Under slightly weaker assumptions than those in Theorem \ref{maintheoremWIENER} we also prove a Lyapunov-type bound on the remainder term in the central limit theorem. Let $\Phi(x)=\int_{-\infty}^x (2 \pi)^{-1/2} e^{-t^2/2} \, \mathrm{d}t$ denote the standard normal distribution function.
\begin{thm}\label{maintheoremCLT} Suppose that $\nu$ is adapted and strictly aperiodic, and that $(\nu^{*k})_{\mathrm{abs}} \neq 0$ for some $k \ge 1$. Let $f: G \to \mathbb{R}$ be Borel measurable such that $\int_G f \, \mathrm{d} \mu =0$, and assume $f \in L^{2+\delta}(G)$ for some $0<\delta \le 1$ and $\sup_{c \in G} \E f(cX_1)^2 < \infty$.
\begin{enumerate}
\item[(i)] If $C(f,\nu) >0$, then for any integer $N \ge N_0(f,\nu,\delta)$
\begin{equation}\label{BerryEsseen}
\sup_{x \in \mathbb{R}} \left| \Pr \left( \frac{\sum_{k=1}^N f(S_k)}{\sqrt{C(f,\nu ) N}} <x \right) - \Phi (x) \right| \ll K \frac{\log^{\delta /(1+\delta)}N}{N^{\delta/( 2+2 \delta )}},
\end{equation}
where $K=\Delta \left( \| f \|_{2+\delta} /\sqrt{C(f, \nu)} \right)^{(2+\delta)/(1+\delta)}$.
\item[(ii)] If $C(f,\nu)=0$, then $N^{-1/2} \sum_{k=1}^N f(S_k) \to 0$ in $L^2$.
\end{enumerate}
\end{thm}

\begin{remark} If $f \in L^3(G)$, then the right hand side of \eqref{BerryEsseen} becomes $KN^{-1/4} \log^{1/2}N$ with $K=\Delta \| f \|_3^{3/2}/C(f,\nu)^{3/4}$. We mention that if $f \in L^4(G)$, then here $\| f \|_3$ can be replaced by $\| f \|_2$ (see the end of Section \ref{section5.2}). As we will see in Proposition \ref{specialcasesprop} below, if $f$ is a class function or $\nu$ is a central measure, then $C(f , \nu) \ge \| f \|_2^2 /(2\Delta)$. Thus in this case $K \le 2 \Delta^{7/4}$, so the right hand side of \eqref{BerryEsseen} does not depend on $f$. ($N_0(f,\nu,\delta)$ always depends on $f$, however.)
\end{remark}

\section{Moment estimates}\label{section3}

Throughout this section we assume that $\nu$ is adapted and strictly aperiodic, and that $(\nu^{*k})_{\mathrm{abs}} \neq 0$ for some $k \ge 1$. Further, we fix a Borel measurable function $f: G \to \mathbb{R}$ such that $\int_G f \, \mathrm{d} \mu =0$, and a uniformly distributed $G$-valued random variable $U$ independent of $X_1, X_2, \ldots$. We now prove moment estimates for the modified sum $\sum_{k=1}^N f(US_k)$. In Section \ref{approxsection} we shall give the counterparts of these estimates for shifted sums $\sum_{k=M+1}^{M+N} f(S_k)$.

For every nonempty, finite interval of positive integers $J \subset \mathbb{N}$ let $S_J=\prod_{j \in J}X_j$. Note that $S_J$ has distribution $\nu^{*|J|}$, hence by definition for any bounded, Borel measurable function $g: G \to \mathbb{R}$
\begin{equation}\label{maintool}
\left| \mathbb{E} g(S_J) - \mathbb{E} g(U) \right| = \left| \int_G g \, \mathrm{d}(\nu^{*|J|}-\mu) \right| \le \sup_G |g| \cdot \Delta_{|J|}.
\end{equation}

\begin{prop}\label{fuskvariance} Assume $f \in L^2(G)$. The series in \eqref{Cf} is absolutely convergent, and $0 \le C(f,\nu) \le \| f \|_2^2 \Delta$. Further, for any integer $N \ge 1$
\[ \begin{split} \E \left( \sum_{k=1}^{N} f(US_k) \right)^2 &= C(f,\nu ) N + O_{\nu} \left( \| f \|_2^2 \right) \\ &\le \| f \|_2^2 \Delta N. \end{split} \]
\end{prop}

\begin{proof} Let $A_k=\E f(U)f(US_k)$. Since $U$ is independent of $S_k$, we have
\begin{equation}\label{EfUfUSk}
A_k = \int_G \int_G f(u)f(ux) \, \mathrm{d} \mu (u) \mathrm{d} \nu^{*k}(x).
\end{equation}
The function $g(x)=\int_G f(u)f(ux) \, \mathrm{d}\mu (u)$ satisfies $\int_G g \, \mathrm{d}\mu=0$ and $\sup_G |g| \le \| f \|_2^2$. Applying \eqref{maintool} to $g$ we thus obtain $\left| A_k \right| \le \| f \|_2^2 \Delta_k$. Since $\Delta_k \to 0$ exponentially fast, the series in \eqref{Cf} is absolutely convergent. As $\E f(U)^2 =\| f \|_2^2$, we have $C(f,\nu) \le \| f \|_2^2 \Delta$. Finally, $C(f,\nu ) \ge 0$ will follow from the second claim.

Expanding the square we have
\begin{equation}\label{fusksum}
\E \left( \sum_{k=1}^{N} f(US_k) \right)^2 = \sum_{k=1}^N \E f(US_k)^2 + 2 \sum_{1 \le k< \ell \le N} \E f(US_k)f(US_{\ell}) .
\end{equation}
Let us write $US_{\ell}=US_kS_{[k+1,\ell]}$. Since $\mu * \nu^{*k}=\mu$, the variable $US_k$ is uniformly distributed on $G$ and independent of $S_{[k+1,\ell]}$; moreover, $S_{[k+1,\ell]} \overset{d}{=} S_{\ell-k}$. Thus $\E f(US_k)^2=\E f(U)^2$ and $\E f(US_k)f(US_{\ell})=\E f(U)f(US_{\ell -k})$, so \eqref{fusksum} simplifies as
\[ \begin{split} \E \left( \sum_{k=1}^{N} f(US_k) \right)^2 &= N \E f(U)^2+2\sum_{1 \le k< \ell \le N} \E f(U)f(US_{\ell -k}) \\ &= N \E f(U)^2 +2 \sum_{d=1}^{N-1} (N-d) A_d \\ &= C(f,\nu ) N +O \left( \sum_{d=1}^{N-1} d |A_d| + N\sum_{d=N}^{\infty} |A_d|  \right) \\ &\le N \E f(U)^2 +2N \sum_{d=1}^{N-1} |A_d| . \end{split} \]
The second claim thus follows from $|A_d|\le \| f \|_2^2 \Delta_d$ and the fact that $\Delta_d \to 0$ exponentially fast.
\end{proof}

We now study the question whether the normalizing factor $C(f,\nu )$ in the variance is zero or positive. To this end, we derive an alternative formula for $C(f,\nu)$ in the form of an infinite series with nonnegative terms. Next, we will consider the special case when $f$ is a class function or $\nu$ is a central measure. As we shall see, the behavior of $C(f,\nu )$ then simplifies considerably, allowing for effective lower bounds.

\begin{prop}\label{Cfalternativeprop} Assume $f \in L^2(G)$. We have
\begin{equation}\label{Cfalternative}
C(f,\nu )= \sum_{\pi \in \hat{G} \backslash \{ \pi_0 \}} d_{\pi} \mathrm{tr} \left( \hat{f}(\pi )^* \hat{f}(\pi ) B_{\nu}(\pi ) \right) ,
\end{equation}
where $B_{\nu}(\pi)=I_{d_{\pi}} + (I_{d_{\pi}} - \hat{\nu} (\pi ))^{-1} \hat{\nu} (\pi ) + (I_{d_{\pi}} - \hat{\nu} (\pi )^*)^{-1} \hat{\nu} (\pi )^*$ and $I_{d_{\pi}}$ denotes the $d_{\pi} \times d_{\pi}$ identity matrix. The series in \eqref{Cfalternative} has nonnegative terms and is convergent. In particular, $C(f,\nu)>0$ if and only if at least one term in \eqref{Cfalternative} is nonzero.
\end{prop}

\begin{proof} Let $A_k=\E f(U)f(US_k)$, and recall \eqref{EfUfUSk}. The function $h_k(u)=\int_G f(ux) \, \mathrm{d}\nu^{*k}(x)$ is in $L^2(G)$, and its Fourier coefficients are $\widehat{h_k}(\pi)=\hat{f}(\pi) \left( \hat{\nu}(\pi)^k \right)^*$. Applying the Parseval formula in \eqref{EfUfUSk} we thus obtain
\begin{equation}\label{akparseval}
A_k = \int_G f(u)h_k(u) \, \mathrm{d} \mu (u) = \sum_{\pi \in \hat{G}} d_{\pi} \mathrm{tr} \left( \hat{f}(\pi) \hat{\nu}(\pi )^k \hat{f}(\pi )^* \right) .
\end{equation}
Here $\hat{f}(\pi _0)=0$ as $\int_G f \, \mathrm{d}\mu =0$. Fix $\pi \in \hat{G}$, $\pi \neq \pi_0$. For any $v \in \mathbb{C}^{d_{\pi}}$ we have $\langle \hat{\nu}(\pi)^k v,v \rangle = \int_G \langle \pi(x)^* v,v \rangle \, \mathrm{d} \nu^{*k}(x)$, where $\langle a,b \rangle=\sum_{i=1}^{d_{\pi}} a_i \overline{b_i}$ denotes the standard sesquilinear form on $\mathbb{C}^{d_{\pi}}$. Since $\pi(x)^*$ is unitary, the integrand $g(x)=\langle \pi(x)^* v,v \rangle$ satisfies $\sup_G |g| \le |v|^2$. Further, $\int_G g \, \mathrm{d} \mu = \langle \hat{\mu}(\pi) v,v \rangle =0$ because $\pi \neq \pi_0$. Applying \eqref{maintool} to $g$ we thus obtain $|\langle \hat{\nu}(\pi)^k v,v \rangle| \le |v|^2 \Delta_k$. In particular, for any $v \in \mathbb{C}^{d_\pi}$ we have $|\langle \hat{f}(\pi) \hat{\nu} (\pi)^k \hat{f}(\pi)^* v,v \rangle| \le \langle \hat{f}(\pi) \hat{f}(\pi)^* v,v \rangle \Delta_k$. Summing this over an orthonormal basis of $\mathbb{C}^{d_{\pi}}$ we get $\left| \mathrm{tr} \left( \hat{f}(\pi) \hat{\nu}(\pi )^k \hat{f}(\pi )^* \right) \right| \le \mathrm{tr} \left( \hat{f}(\pi) \hat{f}(\pi )^* \right) \Delta_k$. Hence
\[ \sum_{k=1}^{\infty} \sum_{\pi \in \hat{G} \backslash \{ \pi_0 \}} d_{\pi} \left| \mathrm{tr} \left( \hat{f}(\pi) \hat{\nu}(\pi )^k \hat{f}(\pi )^* \right) \right| \le \| f \|_2^2 \sum_{k=1}^{\infty} \Delta_k < \infty , \]
justifying a change in the order of summation. Since $\rho (\hat{\nu}(\pi )) \le q<1$, we have $\sum_{k=1}^{\infty} \hat{\nu} (\pi)^k = (I_{d_{\pi}}-\hat{\nu}(\pi))^{-1}\hat{\nu}(\pi)$ in operator norm. We thus obtain
\[ C(f,\nu ) = \| f \|_2^2 + 2 \sum_{k=1}^{\infty} A_k = \sum_{\pi \in \hat{G} \backslash \{ \pi_0 \}} d_{\pi} \mathrm{tr} \left( \hat{f}(\pi )^* \hat{f}(\pi ) \left( I_{d_\pi } + 2 (I_{d_{\pi}} - \hat{\nu} (\pi ))^{-1} \hat{\nu} (\pi ) \right) \right) . \]
As $C(f, \nu )$ is clearly real, we can take the real part of the series in the previous line, resulting in formula \eqref{Cfalternative}.

Finally, we prove that every term in \eqref{Cfalternative} is nonnegative. Fix $\pi \in \hat{G} \backslash \{ \pi_0 \}$. First, suppose that $\pi$ and $\overline{\pi}$ are unitarily inequivalent. Then we may assume that $\pi, \overline{\pi} \in \hat{G}$. Since $f$ and $\nu$ are real-valued, we have $\hat{f}(\overline{\pi})=\overline{\hat{f}(\pi )}$ and $\hat{\nu}(\overline{\pi})=\overline{\hat{\nu}(\pi )}$. Hence $B_{\nu}(\overline{\pi})=\overline{B_{\nu}(\pi )}$, and the terms in \eqref{Cfalternative} indexed by $\pi$ and $\overline{\pi}$ are equal. Let $F$ be the orthogonal projection of $f$ in $L^2(G)$ to the linear subspace spanned by the matrix elements $\{ \pi_{ij} : 1 \le i,j \le d_{\pi} \} \cup \{ \overline{\pi}_{ij} : 1 \le i,j \le d_{\pi} \}$; that is, $F(x)= d_{\pi} \mathrm{tr} \left( \hat{f}(\pi ) \pi (x) \right) + d_{\pi} \mathrm{tr} \left( \hat{f}(\overline{\pi} ) \overline{\pi} (x) \right)$. Note that $F$ is real-valued, $\hat{F}(\pi)=\hat{f}(\pi)$, $\hat{F}(\overline{\pi})=\hat{f}(\overline{\pi})$ and $\hat{F}(\pi')=0$ for all $\pi' \neq \pi, \overline{\pi}$. Therefore the terms in \eqref{Cfalternative} indexed by $\pi$ and $\overline{\pi}$ are both $C(F,\nu)/2$. But $C(F,\nu) \ge 0$ from Proposition \ref{fuskvariance}, and we are done. Next, suppose that $\pi$ and $\overline{\pi}$ are unitarily equivalent. Let $F$ be the orthogonal projection of $f$ in $L^2(G)$ to the linear subspace spanned by the matrix elements $\{ \pi_{ij} : 1 \le i,j \le d_{\pi} \}$; note that $\{ \overline{\pi}_{ij} : 1 \le i,j \le d_{\pi}  \}$ span the same linear subspace. Thus $F(x)= d_{\pi} \mathrm{tr} \left( \hat{f}(\pi ) \pi (x) \right) = d_{\pi} \mathrm{tr} \left( \hat{f}(\overline{\pi} ) \overline{\pi} (x) \right)$. Again, $F$ is real-valued, $\hat{F}(\pi)=\hat{f}(\pi)$ and $\hat{F}(\pi')=0$ for all $\pi' \neq \pi$. Therefore the term in \eqref{Cfalternative} indexed by $\pi$ is $C(F,\nu) \ge 0$.
\end{proof}

\begin{prop}\label{specialcasesprop} Assume $f \in L^2(G)$, and let $\nu^*(B)=\nu(B^{-1})$ ($B \subseteq G$ Borel) denote the distribution of $X_1^{-1}$. Suppose at least one of the following hold.
\begin{enumerate}
\item[(i)] $f$ is a class function
\item[(ii)] $\nu * \nu^* = \nu^* * \nu$
\item[(iii)] $\nu$ is a central measure
\end{enumerate}
Then $\frac{1-q}{1+q} \| f \|_2^2 \le C(f,\nu) \le \frac{1+q}{1-q} \| f \|_2^2$. In particular, $C(f,\nu)=0$ if and only if $f=0$ $\mu$-a.e.
\end{prop}

\begin{proof} First, assume (i). It follows from Schur's lemma that $\hat{f}(\pi)$ is a scalar multiple of the identity matrix. Hence \eqref{Cfalternative} simplifies as
\[ C(f,\nu ) = \sum_{\pi \in \hat{G} \backslash \{ \pi_0 \}} \mathrm{tr} \left( \hat{f}(\pi )^* \hat{f}(\pi ) \right) \mathrm{tr}\, B_{\nu}(\pi ) . \]
Let $\lambda_1, \lambda_2, \dots, \lambda_{d_{\pi}}$ denote the eigenvalues of $\hat{\nu}(\pi)$. Then
\[ \mathrm{tr}\, B_{\nu}(\pi ) = \sum_{i=1}^{d_{\pi}} \left( 1 + \frac{\lambda_i}{1-\lambda_i} + \frac{\overline{\lambda_i}}{1-\overline{\lambda_i}} \right) = \sum_{i=1}^{d_{\pi}} \frac{1-|\lambda_i|^2}{|1-\lambda_i|^2} . \]
Since $\rho (\hat{\nu}(\pi)) \le q<1$, we have $|\lambda_i| \le q$, and so $d_{\pi}\frac{1-q}{1+q} \le \mathrm{tr}\, B_{\nu} (\pi) \le d_{\pi} \frac{1+q}{1-q}$. The claim thus follows from the Parseval formula.

Next, assume (ii). Since $\widehat{\nu^*}(\pi) = \hat{\nu}(\pi)^*$, the condition $\nu * \nu^* = \nu^* *\nu$ implies that the matrix $\hat{\nu}(\pi)$ is normal. Therefore there exists an orthonormal basis $v_1, v_2, \dots, v_{d_{\pi}}$ of $\mathbb{C}^{d_{\pi}}$ comprised of eigenvectors of $\hat{\nu}(\pi)$; say, $\hat{\nu}(\pi)v_i=\lambda_i v_i$. It follows that $\hat{\nu}(\pi)^* v_i = \overline{\lambda_i} v_i$, and hence
\[ B_{\nu}(\pi) v_i = \left( 1 + \frac{\lambda_i}{1-\lambda_i} + \frac{\overline{\lambda_i}}{1-\overline{\lambda_i}} \right) v_i = \frac{1-|\lambda_i|^2}{|1-\lambda_i|^2} v_i . \]
The eigenvalues of $B_{\nu}(\pi)$ again satisfy $\frac{1-q}{1+q} \le \frac{1-|\lambda_i|^2}{|1-\lambda_i|^2} \le \frac{1+q}{1-q}$. It is now easy to see that
\[ \frac{1-q}{1+q} \mathrm{tr} \left( \hat{f}(\pi )^* \hat{f}(\pi ) \right) \le \mathrm{tr} \left( \hat{f}(\pi )^* \hat{f}(\pi ) B_{\nu}(\pi ) \right) \le \frac{1+q}{1-q} \mathrm{tr} \left( \hat{f}(\pi )^* \hat{f}(\pi ) \right) , \]
and the claim follows. Finally, note that condition (iii) implies condition (ii).
\end{proof}

We conclude this section with an estimate of the $L^p$-norm for $1 \le p \le 4$. These estimates, combined with the Erd\H{o}s--Stechkin and the Rademacher--Menshov inequalities will help us bound the fluctuations of $\sum_{k=1}^N f(S_k)$ as $N$ runs in a short interval. Additionally, we will also use them to verify the Lyapunov condition in the proof of Theorem \ref{maintheoremCLT}.

\begin{prop}\label{highermoments} Assume $f \in L^p(G)$ for some $1 \le p \le 4$. For any integer $N \ge 1$
\[ \left\| \sum_{k=1}^{N} f(US_k) \right\|_p \ll \left\{ \begin{array}{ll} \| f \|_p \left( \Delta N \right)^{1/p} & \mathrm{if} \,\,\, 1 \le p < 2, \\ \| f \|_p \sqrt{\Delta N} & \mathrm{if} \,\,\, 2 \le p \le 4. \end{array} \right. \]
In the case $p=4$ we also have
\begin{equation}\label{fusk4thmoment}
\left\| \sum_{k=1}^{N} f(US_k) \right\|_4 \ll \| f \|_2 \sqrt{\Delta N} + \| f \|_4 \Delta^{3/4} N^{1/4}.
\end{equation}
\end{prop}

\begin{proof} First, assume $p=4$. Expanding the fourth power we get
\begin{equation}\label{expandfourth}
\E \left( \sum_{k=1}^{N} f(US_k) \right)^4 \ll \sum_{1 \le k_1 \le k_2 \le k_3 \le k_4 \le N} \left| \E f(US_{k_1})f(US_{k_2})f(US_{k_3})f(US_{k_4}) \right| .
\end{equation}
Fix $1 \le k_1 \le k_2 \le k_3 \le k_4 \le N$. Since $US_{k_1}$ is uniformly distributed on $G$ and independent of $X_{k_1+1}, X_{k_1+2}, \ldots$, we have
\[ \begin{split} &\E f(US_{k_1})f(US_{k_2})f(US_{k_3})f(US_{k_4}) = \\ &\int_G \int_G \int_G \int_G f(u)f(ux)f(uxy)f(uxyz) \, \mathrm{d} \mu (u) \mathrm{d} \nu^{*(k_2-k_1)} (x) \mathrm{d} \nu^{*(k_3-k_2)} (y) \mathrm{d} \nu^{*(k_4-k_3)} (z) . \end{split} \]
Here we use the convention that $\nu^{*0}$ is the Dirac measure concentrated on the identity element of $G$, and $\Delta_0=\| \nu^{*0}-\mu \|_{\mathrm{TV}} \le 2$. Let
\[ g(z)= \int_G \int_G \int_G f(u)f(ux)f(uxy)f(uxyz) \, \mathrm{d} \mu (u) \mathrm{d} \nu^{*(k_2-k_1)} (x) \mathrm{d} \nu^{*(k_3-k_2)} (y) . \]
As $\int_G g(z) \, \mathrm{d} \mu (z) =0$, applying \eqref{maintool} to $g$ we obtain
\begin{equation}\label{k1k2k3k4}
\left| \E f(US_{k_1})f(US_{k_2})f(US_{k_3})f(US_{k_4}) \right| \le \sup_G |g| \cdot \Delta_{k_4-k_3}.
\end{equation}
Fix $z \in G$, and let $h_z(y)= \int_G \int_G f(u)f(ux)f(uxy)f(uxyz) \, \mathrm{d} \mu (u) \mathrm{d} \nu^{*(k_2-k_1)} (x)$. Note that
\[ \int_G h_z(y) \, \mathrm{d} \mu (y) = \int_G \int_G f(u)f(ux) w_z \, \mathrm{d} \mu (u) \mathrm{d} \nu^{*(k_2-k_1)}(x) ,  \]
where $w_z=\int_G f(uxy)f(uxyz) \mathrm{d} \mu (y) = \int_G f(y) f(yz) \mathrm{d} \mu (y)$ does not depend on $u$ and $x$. Applying \eqref{maintool} to $h_z$ we get
\[ \begin{split} |g(z)| &= \left| \int_G h_z(y) \, \mathrm{d} \nu^{*(k_3-k_2)}(y) \right| \\ &\le |w_z| \cdot \left| \int_G \int_G f(u)f(ux) \, \mathrm{d} \mu (u) \mathrm{d} \nu^{*(k_2-k_1)}(x) \right| + \sup_G |h_z| \cdot \Delta_{k_3-k_2} . \end{split} \]
Here $|w_z| \le \| f \|_2^2$, and the double integral in the previous line is $\le \| f \|_2^2 \Delta_{k_2-k_1}$, as seen in the proof of Proposition \ref{fuskvariance}. Hence
\begin{equation}\label{gz}
|g(z)| \le \| f \|_2^4 \Delta_{k_2-k_1}+\sup_G |h_z| \cdot \Delta_{k_3-k_2}.
\end{equation}
Now fix $y,z \in G$, and let $r_{y,z}(x)=\int_G f(u)f(ux)f(uxy)f(uxyz) \, \mathrm{d} \mu (u)$. Note that $\sup_G |r_{y,z}| \le \| f \|_4^4$, and that $\int_G r_{y,z}(x) \, \mathrm{d} \mu (x) =0$. Applying \eqref{maintool} we thus get
\begin{equation}\label{hzy}
\left| h_z(y) \right| = \left| \int_G r_{y,z}(x) \, \mathrm{d} \nu^{*(k_2-k_1)}(x) \right| \le \| f \|_4^4 \Delta_{k_2-k_1}.
\end{equation}
Combining \eqref{k1k2k3k4}, \eqref{gz} and \eqref{hzy} we finally obtain
\begin{multline*} \left| \E f(US_{k_1})f(US_{k_2})f(US_{k_3})f(US_{k_4}) \right| \le \\ \| f \|_2^4 \Delta_{k_2-k_1} \Delta_{k_4-k_3} + \| f \|_4^4 \Delta_{k_2-k_1} \Delta_{k_3-k_2} \Delta_{k_4-k_3} . \end{multline*}
Summing over $1 \le k_1 \le k_2 \le k_3 \le k_4 \le N$, \eqref{fusk4thmoment} follows.

On the other hand, we can use $\Delta_{k_3-k_2} \le 2$ to deduce the simpler estimate
\[ \left| \E f(US_{k_1})f(US_{k_2})f(US_{k_3})f(US_{k_4}) \right| \le 3 \| f \|_4^4 \Delta_{k_2-k_1} \Delta_{k_4-k_3} , \]
and by summing over $1 \le k_1 \le k_2 \le k_3 \le k_4 \le N$ we get $\left\| \sum_{k=1}^{N} f(US_k) \right\|_4 \ll \| f \|_4 \sqrt{\Delta N}$. Proposition \ref{fuskvariance} shows that if $f \in L^2(G)$, the same estimate holds with $\| \cdot \|_4$ replaced by $\| \cdot \|_2$ on both sides. Moreover, we also have the trivial estimate $\left\| \sum_{k=1}^N f(US_k) \right\|_1 \le \| f \|_1 N$ for any $f\in L^1(G)$. This settles the endpoints of the intervals $1 \le p \le 2$ and $2 \le p \le 4$. The cases $1<p<2$ and $2<p<4$ follow from the Riesz--Thorin interpolation theorem applied to the linear operator $f \mapsto \sum_{k=1}^N f(US_k) - N \int_G f \, \mathrm{d} \mu$.
\end{proof}

\section{Approximation by independent variables}\label{approxsection}

Assume again, that $\nu$ is adapted and strictly aperiodic, and that $(\nu^{*k})_{\mathrm{abs}} \neq 0$ for some $k \ge 1$. Fix a Borel measurable function $f: G \to \mathbb{R}$ such that $\int_G f \, \mathrm{d} \mu =0$. In this section we approximate the shifted sum $\sum_{k=M+1}^{M+N} f(S_k)$ by a sum of independent random variables. The main tool of this approximation is a coupling between $\nu^{*k}$ and $\mu$, which we will construct using Strassen's theorem. We mention that this is the only step of the proof where we use the fact that $G$ is metrizable. A similar approach was used by Schatte \cite{SCH} on the circle group $G = \mathbb{R} / \mathbb{Z}$, with a different type of coupling based on the Kolmogorov metric instead of the total variation metric.

Recall that for any two Borel probability measures $\nu_1$ and $\nu_2$ on $G$ (or indeed, on any Polish space) we have $\| \nu_1 - \nu_2 \|_{\mathrm{TV}} = 2 \inf_{\vartheta} \vartheta \left( \{ (x,y) \in G \times G : x \neq y \} \right)$, where the infimum is over all Borel probability measures $\vartheta$ on $G \times G$ whose marginals are $\vartheta (B \times G) = \nu_1 (B)$ and $\vartheta (G \times B) = \nu_2 (B)$. This fact follows from Strassen's theorem, which in turn is a special case of the Kantorovich duality theorem in the theory of optimal transportation (see e.g.\ \cite[Chapter 1]{V}). In particular, for any $k \ge 1$ there exists a Borel probability measure $\vartheta_k$ on $G \times G$ with marginals $\nu^{*k}$ and $\mu$ such that $\vartheta_k \left( \{ (x,y) \in G \times G : x \neq y \} \right) \le \Delta_{k}$. After a suitable extension of the probability space, for any nonempty, finite interval of positive integers $J \subseteq \mathbb{N}$ we may therefore introduce auxiliary $G$-valued random variables $T_J, U_J$ whose joint distribution is $\vartheta_{|J|}$; that is, $T_J \overset{d}{=} S_J$, $U_J$ is uniformly distributed on $G$, and $\Pr (T_J \neq U_J) \le \Delta_{|J|}$. Moreover, we may assume $(T_J, U_J)$, $J \subseteq \mathbb{N}$ and $X_1, X_2, \ldots$ are independent. The independence of the approximating variables will follow from the following observation.

\begin{lem}\label{schattelemma} Let $G$ be a compact metrizable group, and let $(S,\mathcal{A})$ be a measurable space. Let $U$ be a $G$-valued, and let $V$ be an $S$-valued random variable. If $U$ and $V$ are independent and $U$ is uniformly distributed on $G$, then for any Borel measurable function $g: S \to G$ the variables $g(V) U$ and $V$ are also independent.
\end{lem}

\begin{proof} Note that $g(V) U$ is uniformly distributed on $G$. Let $\gamma$ denote the distribution of $V$. For any Borel set $B \subseteq G$ and any $A \in \mathcal{A}$ we have
\[ \begin{split} \Pr \left( g(V) U \in B, V \in A \right) &= \int_S \int_G I_B (g(v)u) I_A (v) \, \mathrm{d} \mu (u) \mathrm{d} \gamma (v) \\ &= \int_S \int_G I_{g(v)^{-1}B} (u) I_A (v) \, \mathrm{d} \mu (u) \mathrm{d} \gamma (v) \\ &= \int_S \mu (B) I_A (v) \, \mathrm{d} \gamma (v) \\ &= \Pr \left( g(V)U \in B \right) \Pr (V \in A) .  \end{split} \]
\end{proof}

We construct the approximating variables as follows. Fix an integer $M \ge 0$, and let us decompose the infinite set $\{ M+1, M+2, \ldots \}$ into consecutive, nonempty, finite intervals of integers $H_1, J_1, H_2, J_2, \ldots$. For all $i \ge 1$ and $k \in J_i$ let
\[ W_k=S_M \prod_{j=1}^{i-1} \left( T_{H_j} S_{J_j} \right) T_{H_i} \prod_{\substack{\ell \in J_i \\ \ell \le k}} X_{\ell}, \quad W_k^*=S_M \prod_{j=1}^{i-1} \left( T_{H_j} S_{J_j} \right) U_{H_i} \prod_{\substack{\ell \in J_i \\ \ell \le k}} X_{\ell} . \]
Similarly, for all $i \ge 2$ and $k \in H_i$ let
\[ W_k=S_M \prod_{j=1}^{i-2} \left( S_{H_j} T_{J_j} \right) S_{H_{i-1}} T_{J_{i-1}} \prod_{\substack{\ell \in H_i \\ \ell \le k}} X_{\ell}, \quad W_k^*= S_M \prod_{j=1}^{i-2} \left( S_{H_j} T_{J_j} \right) S_{H_{i-1}} U_{J_{i-1}} \prod_{\substack{\ell \in H_i \\ \ell \le k}} X_{\ell} . \]
Note that here the case $i=1$ is excluded to ensure that $H_i$ is preceded by an interval $J_{i-1}$. Let us also introduce the variables
\[ \begin{split} Y_i &= \sum_{k \in J_i} f(W_k), \qquad Y_i^* = \sum_{k \in J_i} f(W_k^*) \quad (i \ge 1), \\ Z_i &= \sum_{k \in H_i} f(W_k), \qquad Z_i^* = \sum_{k \in H_i} f(W_k^*) \quad (i \ge 2) . \end{split} \]
Observe that the random sequence $W_k$, $k \in \bigcup_{i=1}^{\infty} J_i$ has the same distribution as $S_k$, $k \in \bigcup_{i=1}^{\infty} J_i$. Similarly, $W_k$, $k \in \bigcup_{i=2}^{\infty} H_i$ has the same distribution as $S_k$, $k \in \bigcup_{i=2}^{\infty} H_i$.

For every $R \ge 1$ let $N_R$ be such that $M+N_R=\max J_R$. Then $\sum_{k=M+1}^{M+N} f(S_k)$ along the subsequence $N_R$ satisfies
\[ \sum_{k=M+1}^{M+N_R} f(S_k) = \sum_{k \in H_1} f(S_k) + \sum_{i=1}^R \sum_{k \in J_i} f(S_k) + \sum_{i=2}^R \sum_{k \in H_i} f(S_k) . \]
Here the sequence $\sum_{i=1}^R \sum_{k \in J_i} f(S_k)$, $R=1,2,\dots$ has the same distribution as $\sum_{i=1}^R Y_i$, $R=1,2,\dots$; similarly, the sequence $\sum_{i=2}^R \sum_{k \in H_i} f(S_k)$, $R=2,3,\dots$ has the same distribution as $\sum_{i=2}^R Z_i$, $R=2,3,\dots$. The main idea is to replace $Y_i$ by $Y_i^*$, and $Z_i$ by $Z_i^*$. First, we establish the properties of the approximating variables $Y_i^*$ and $Z_i^*$, then we estimate the error committed.

\begin{lem}\label{yi*zi*} $Y_1^*, Y_2^*, \ldots$ are independent, and $\E Y_i^*=0$.
\begin{enumerate}
\item[(i)] If $f \in L^2(G)$, then $\E (Y_i^*)^2 = C(f,\nu) |J_i| +O_{\nu}(\| f \|_2^2)$.
\item[(ii)] If $f \in L^p(G)$ for some $1 \le p \le 4$, then for any $0 \le R<S$
\[ \left\| \sum_{i=R+1}^S Y_i^* \right\|_p \ll \left\{ \begin{array}{ll} \| f \|_p \left( \Delta \sum_{i=R+1}^S |J_i| \right)^{1/p} & \mathrm{if} \,\,\, 1 \le p < 2, \\ \| f \|_p \sqrt{\Delta \sum_{i=R+1}^S |J_i|} & \mathrm{if} \,\,\, 2 \le p \le 4. \end{array} \right.  \]
In the case $p=4$ we also have
\begin{equation}\label{yi*4thmoment}
\left\| \sum_{i=R+1}^S Y_i^* \right\|_4 \ll \| f \|_2 \sqrt{\Delta \sum_{i=R+1}^S |J_i|} + \| f \|_4 \Delta^{3/4} \left( \sum_{i=R+1}^S|J_i| \right)^{1/4}.
\end{equation}
\end{enumerate}
The same hold for $Z_2^*, Z_3^*, \ldots$ with $|J_i|$ replaced by $|H_i|$.
\end{lem}

\begin{proof} To see that $Y_1^*, Y_2^*, \ldots$ are independent, it will be enough to prove that $Y_i^*$ is independent of the random vector $(Y_1^*, Y_2^*, \dots, Y_{i-1}^*)$ for all $i \ge 2$. Let $W$ be the random vector whose coordinates are the variables $X_k$, $k \in [1,M] \cup J_1 \cup \cdots \cup J_{i-1}$ and $T_{H_1}, U_{H_1}, T_{H_2}, U_{H_2}, \dots, T_{H_{i-1}}, U_{H_{i-1}}$. Further, let $W'$ be the random vector with coordinates $X_k$, $k \in J_i$. Applying Lemma \ref{schattelemma} to $V=(W,W')$ and $U=U_{H_i}$ we get that $(W,W')$ and $g(W,W')U_{H_i}$ are independent for any Borel measurable function $g$. But $W$ and $W'$ are also independent, therefore $W$, $W'$, $g(W,W')U_{H_i}$ are independent as well. Note that $(Y_1^*, Y_2^*, \dots, Y_{i-1}^*)$ is a function of $W$, whereas $Y_i^*$ is a function of $W'$ and $g(W,W')U_{H_i}$ for some $g$ (in fact, $g(W,W')$ is simply the product of certain components of $W$). The independence thus follows.

Now fix $i \ge 1$. Note that $S_M \prod_{j=1}^{i-1} \left( T_{H_j} S_{J_j} \right) U_{H_i}$ is uniformly distributed on $G$ and independent of $X_k$, $k \in J_i$. Hence $Y_i^*=\sum_{k \in J_i} f(W_k^*) \overset{d}{=} \sum_{k=1}^{|J_i|} f(US_k)$. Here $US_k$ is uniformly distributed on $G$; in particular, $\E Y_i^* = \sum_{k=1}^{|J_i|} \E f(US_k)=0$.

Claim (i) follows from Proposition \ref{fuskvariance}. Now fix $0 \le R <S$, and let us prove (ii). The case $p=1$ follows from $\| Y_i^* \|_1 \le \sum_{k=1}^{|J_i|} \| f(US_k) \|_1 = \| f \|_1 |J_i|$. If $p=2$, Proposition \ref{fuskvariance} gives $\| Y_i^* \|_2 = \left\| \sum_{k=1}^{|J_i|} f(US_k) \right\|_2 \le \| f \|_2 \sqrt{\Delta |J_i|}$, hence the claim follows from independence. Now assume $p=4$. The independence of $Y_1^*, Y_2^*, \dots$ implies
\[ \E \left| \sum_{i=R+1}^S Y_i^* \right|^4 \ll \sum_{i=R+1}^S \E |Y_i^*|^4 + \left( \sum_{i=R+1}^S \E |Y_i^*|^2 \right)^2 . \]
Proposition \ref{highermoments} shows $\| Y_i^* \|_4 = \left\| \sum_{k=1}^{|J_i|} f(US_k) \right\|_4 \ll \| f \|_2 \sqrt{\Delta |J_i|} + \| f \|_4 \Delta^{3/4} |J_i|^{1/4}$, yielding \eqref{yi*4thmoment}. On the other hand, Proposition \ref{highermoments} also gives $\| Y_i^* \|_4 \ll \| f \|_4 \sqrt{\Delta |J_i|}$, and so $\left\| \sum_{i=R+1}^S Y_i^* \right\|_4 \ll \| f \|_4 \sqrt{\Delta \sum_{i=R+1}^S |J_i|}$ follows as well. This settles the endpoints of the intervals $1 \le p \le 2$ and $2 \le p \le 4$.

Observe that for a given integer $M \ge 0$, given intervals $H_1, J_1, \dots$ and given $0 \le R<S$, the sum $\sum_{i=R+1}^S Y_i^*$ is linear in $f$. Applying the Riesz--Thorin interpolation theorem to the linear operator $f \mapsto \sum_{i=R+1}^S Y_i^* - \left( \sum_{i=R+1}^S |J_i| \right) \int_G f \, \mathrm{d}\mu$, the cases $1<p<2$ and $2<p<4$ follow. The proof for $Z_2^*, Z_3^*, \ldots$ is analogous.
\end{proof}

\begin{lem}\label{approxerrorlemma} If $L_p=\sup_{c \in G} \E |f(cX_1)|^p< \infty$ for some $p \ge 1$, then $\| Y_i-Y_i^* \|_p \le 2 |J_i| \left( L_p \Delta_{|H_i|} \right)^{1/p}$ and $\| Z_i-Z_i^* \|_p \le 2 |H_i| \left( L_p \Delta_{|J_i|} \right)^{1/p}$.

\end{lem}

\begin{proof} We have $\| Y_i-Y_i^* \|_p \le \sum_{k \in J_i} \| f(W_k)-f(W_k^*) \|_p$. Let $\mathcal{F}$ be the $\sigma$-algebra generated by $S_M \prod_{j=1}^{i-1} \left( T_{H_j} S_{J_j} \right)$, $T_{H_i}$, $U_{H_i}$ and $X_{\ell}$, $\ell \in J_i$, $\ell <k$. Then $W_k=aX_k$ and $W_k^*=a^*X_k$ with some $\mathcal{F}$-measurable random variables $a,a^*$. Note that if $T_{H_i}=U_{H_i}$, then $W_k=W_k^*$. Therefore $\E \left( \left| f(W_k)-f(W_k^*)\right|^p \mid \mathcal{F} \right) \le 2^p L_p I_{\{ T_{H_i} \neq U_{H_i} \}}$. Taking the (total) expectation we get $\E \left| f(W_k)-f(W_k^*) \right|^p \le 2^p L_p \Pr (T_{H_i} \neq U_{H_i})$ $\le 2^p L_p \Delta_{|H_i|}$, and the result follows. The proof for $\| Z_i-Z_i^* \|_p$ is analogous.
\end{proof}

As a simple application of the approximating variables constructed above, we deduce moment estimates for shifted sums $\sum_{k=M+1}^{M+N} f(S_k)$ from the results of Section \ref{section3}.

\begin{cor}\label{momentscorollary}
\
\begin{enumerate}
\item[(i)] If $\sup_{c \in G} \E f(cX_1)^2 < \infty$, then for any integers $M \ge 0$ and $N \ge 1$
\[ \left\| \sum_{k=M+1}^{M+N} f(S_k) \right\|_2 = \sqrt{C(f,\nu ) N} +O_{f,\nu} \left( \log (N+1) \right) . \]
\item[(ii)] If $\sup_{c \in G} \E \left| f(cX_1)\right|^p <\infty$ for some $1 \le p \le 4$, then for any integers $M \ge 0$ and $N \ge 1$
\[ \left\| \sum_{k=M+1}^{M+N} f(S_k) \right\|_p \ll K_{f,\nu ,p} \log (N+1) + \left\{ \begin{array}{ll} \| f \|_p \left( \Delta N \right)^{1/p} & \mathrm{if} \,\,\, 1 \le p < 2, \\ \| f \|_p \sqrt{\Delta N} & \mathrm{if} \,\,\, 2 \le p \le 4. \end{array} \right. \]
with some constant $K_{f,\nu,p}>0$. In the case $p=4$ we also have
\[ \left\| \sum_{k=M+1}^{M+N} f(S_k) \right\|_4 \ll K_{f, \nu, p} \log (N+1)+ \| f \|_2 \sqrt{\Delta N} + \| f \|_4 \Delta^{3/4} N^{1/4}. \]
\end{enumerate}
\end{cor}

\begin{proof} We may assume that $N$ is large enough in terms of $f$, $\nu$ and $p$. Let us decompose the index set $[M+1,M+N]$ into two consecutive intervals of integers $H_1$ and $J_1$ such that $|H_1|=\lceil 4 \Delta \log N \rceil$. We then have $\sum_{k=M+1}^{M+N} f(S_k) = \sum_{k \in H_1} f(S_k)+\sum_{k \in J_1} f(S_k)$, where $\sum_{k \in J_1} f(S_k) \overset{d}{=} Y_1$. To see (i), let us write
\[ \left\| \sum_{k=M+1}^{M+N} f(S_k) \right\|_2 = \| Y_1^* \|_2 +O \left( \| Y_1-Y_1^* \|_2 + \sum_{k \in H_1} \| f(S_k) \|_2  \right) . \]
By Lemma \ref{yi*zi*} (i), here $\| Y_1^* \|_2 = \sqrt{C(f,\nu ) N} +O_{f,\nu} (1)$. Since, say, $\Delta_k^{1/k} \le (1+q)/2$ for $k \ge k_0(\nu)$ and $((1+q)/2)^{\Delta} \le ((1+q)/2)^{1/(1-q)} \le e^{-1/2}$, we have $\Delta_{|H_1|} \le N^{-2}$. Lemma \ref{approxerrorlemma} thus gives $\| Y_1-Y_1^* \|_2 \ll_{f,\nu} 1$. Finally, note that $\sup_{c \in G} \E f(cX_1)^2 < \infty$ implies $\sup_{k \ge 1} \E f(S_k)^2 < \infty$. Hence $\sum_{k \in H_1} \| f(S_k) \|_2 \ll_{f,\nu} \log N$, and (i) follows. If we use Lemma \ref{yi*zi*} (ii) instead of Lemma \ref{yi*zi*} (i), similar arguments show (ii).
\end{proof}

\begin{remark} We could easily improve the error term $K_{f,\nu ,p} \log (N+1)$ in (ii) by decomposing $[M+1,M+N]$ into more than $2$ consecutive intervals of exponentially increasing sizes.
\end{remark}

\section{Proof of the theorems}

We prove Theorem \ref{maintheoremLIL} (i) for strictly aperiodic measures and Theorem \ref{maintheoremWIENER} in Section \ref{section5.1}; the general case of Theorem \ref{maintheoremLIL} and Theorem \ref{stronguniformityLIL} in Section \ref{section5.2}; finally, Theorem \ref{maintheoremCLT} and the Remark thereafter, and Theorem \ref{stronguniformityCLT} in Section \ref{section5.3}.

\subsection{Almost sure asymptotics, strictly aperiodic case}\label{section5.1}

Suppose that $\nu$ is adapted and strictly aperiodic, and that $(\nu^{*k})_{\mathrm{abs}} \neq 0$ for some $k \ge 1$. Let $f: G \to \mathbb{R}$ be Borel measurable such that $\int_G f \, \mathrm{d} \mu =0$, and assume $\sup_{c \in G} \E |f(cX_1)|^p <\infty$. In this section we prove the strong law of large numbers \eqref{mainSLLN} in the case $1 \le p \le 2$, and the almost sure approximation by a Wiener process in Theorem \ref{maintheoremWIENER} in the case $2<p<4$. For the sake of brevity, in the proofs of this section implied constants are allowed to depend on $f$, $\nu$ and $p$.

First, assume $1 \le p \le 2$. We start by estimating the fluctuations. Recall that $\log_m$ denotes the $m$-fold iterated logarithm.
\begin{lem}\label{fluctuationlemma} For any integers $m\ge 1$, $M \ge 0$ and $N \ge N_0(f, \nu, p, m)$, and any $\lambda>\lambda_0(f, \nu, p, m)$
\[ \Pr \left( \max_{1 \le n \le N} \left| \sum_{k=M+1}^{M+n} f(S_k) \right| \ge \lambda N^{1/p} \right) \ll_{f, \nu, p, m} \frac{\left( \log_m N \right)^p}{\lambda^p} . \]
\end{lem}

\begin{proof} We use induction on $m$. Corollary \ref{momentscorollary} (ii) and the Rademacher--Menshov inequality \cite[Theorem F]{MO} give $\left\| \max_{1 \le n \le N} \left| \sum_{k=M+1}^{M+n} f(S_k) \right| \right\|_p \ll N^{1/p} \log (N+1)$. The $m=1$ case thus follows from the Markov inequality. Now assume the claim holds for some $m \ge 1$, and let us prove it for $m+1$. Let us decompose the index set $[M+1,M+N]$ into consecutive intervals of integers $H_1, J_1, H_2, J_2, \dots, H_R, J_R$, as in Section \ref{approxsection}, such that $|H_i|, |J_i| \ge 4 \Delta \log N$ for all $i$, and $R=\Theta (N/\log N)$. Similarly to the proof of Corollary \ref{momentscorollary} we have $\Delta_{|H_i|}, \Delta_{|J_i|} \le N^{-2}$. Let $M+n_r=\max J_r$, and recall that for any $2 \le r \le R$
\[ \sum_{M+1}^{M+n_r} f(S_k) = \sum_{k \in H_1 \cup J_1} f(S_k) + \sum_{i=2}^r \sum_{k \in J_i} f(S_k) + \sum_{i=2}^r \sum_{k \in H_i} f(S_k) . \]
Here the variables $\sum_{i=2}^r \sum_{k \in J_i} f(S_k)$, $2 \le r \le R$ have the same joint distribution as $\sum_{i=2}^r Y_i$, $2 \le r \le R$; similarly, $\sum_{i=2}^r \sum_{k \in H_i} f(S_k)$, $2 \le r \le R$ have the same joint distribution as $\sum_{i=2}^r Z_i$, $2 \le r \le R$. Let us introduce the random events
\[ \begin{split} A&=\left\{ \max_{2 \le r \le R} \left| \sum_{i=1}^r Y_i \right| \ge \frac{\lambda}{4} N^{1/p} \right\}, \\ B&=\left\{ \max_{2 \le r \le R} \left| \sum_{i=2}^r Z_i \right| \ge \frac{\lambda}{4} N^{1/p} \right\}, \\ C_i&= \left\{ \max_{n \in H_i \cup J_i} \left| \sum_{k=\min H_i}^{n} f(S_k) \right| \ge \frac{\lambda}{4} N^{1/p} \right\} . \end{split} \]
The event in the claim of the lemma is a subset of $A \cup B \cup \bigcup_{i=1}^R C_i$. Applying the inductive hypothesis on the interval $H_i \cup J_i$ of length $\ll \log N$ we get $\Pr (C_i) \ll \lambda^{-p} (\log N /N) (\log_m \log N)^p$, and hence $\Pr \left( \bigcup_{i=1}^R C_i \right) \ll \lambda^{-p} (\log_{m+1}N)^p$.

Recall from Lemma \ref{yi*zi*} (ii) that $\left\| \sum_{i=r+1}^s Y_i^* \right\|_p \ll \left( \sum_{i=r+1}^s |J_i| \right)^{1/p} \ll N^{1/p}$ for any $1 \le r <s \le R$. It follows that the median of $\sum_{i=r+1}^s Y_i^*$ is also $\ll N^{1/p}$, and hence from L\'evy's inequality (see e.g.\ \cite[p.\ 51]{P}) we get
\begin{equation}\label{levyinequality}
\Pr \left( \max_{2 \le r \le R} \left| \sum_{i=2}^r Y_i^* \right| \ge \frac{\lambda}{8} N^{1/p} \right) \le 2 \Pr \left( \left| \sum_{i=2}^R Y_i^* \right| \ge \frac{\lambda}{16} N^{1/p} \right) \ll \frac{1}{\lambda^p}
\end{equation}
provided $\lambda$ is large enough. On the other hand, Lemma \ref{approxerrorlemma} gives $\E |Y_i-Y_i^*|^p \ll |J_i|^p \Delta_{|H_i|} \ll N^{-2}(\log N)^p$,
and thus
\[ \Pr \left( |Y_i-Y_i^*| \ge \frac{\lambda}{8R} N^{1/p} \right) \ll \frac{(\log N)^p}{N^2} \cdot \frac{R^p}{\lambda^p N} \ll \frac{1}{\lambda^p N}. \]
Therefore $\Pr \left( \sum_{i=2}^R |Y_i-Y_i^*| \ge (\lambda /8) N^{1/p} \right) \ll \lambda^{-p}$. This relation, together with \eqref{levyinequality} shows $\Pr (A) \ll \lambda^{-p}$. Repeating the same arguments for $Z_i$ and $Z_i^*$ we get $\Pr (B) \ll \lambda^{-p}$. Hence $\Pr \left( A \cup B \cup \bigcup_{i=1}^R C_i \right) \ll \lambda^{-p} (\log_{m+1} N)^p$, as claimed.
\end{proof}

We are now ready to prove \eqref{mainSLLN}. Fix $m \ge 1$ and $\varepsilon >0$. Let us decompose the set of positive integers into consecutive intervals of integers $H_1, J_1, H_2, J_2, \dots$, as in Section \ref{approxsection} (with the choice $M=0$), such that, say, $|H_i|=|J_i|=i$ for all $i \ge 1$. Similarly to the proof of Corollary \ref{momentscorollary} we have $i \ge 16 \Delta \log i$, and so $\Delta_{|H_i|}=\Delta_{|J_i|} \le i^{-8}$ for all integers $i$ large enough in terms of $\nu$.

Consider \eqref{mainSLLN} along the subsequence $N_R = \max J_R = \Theta (R^2)$. We have
\[ \sum_{k=1}^{N_R} f(S_k) = \sum_{k \in H_1 \cup J_1} f(S_k) + \sum_{i=2}^R \sum_{k \in J_i} f(S_k) + \sum_{i=2}^R \sum_{k \in H_i} f(S_k) . \]
Here the sequences $\sum_{i=2}^R \sum_{k \in J_i} f(S_k)$ and $\sum_{i=2}^R \sum_{k \in H_i} f(S_k)$, $R=2,3,\dots$ have the same distribution as $\sum_{i=2}^R Y_i$ and $\sum_{i=2}^R Z_i$, $R=2,3,\dots$, respectively. Using Lemma \ref{yi*zi*} (ii) we get $\sum_{i=i_0(m)}^{\infty} \E |Y_i^*|^p/(i \varphi_{m, \varepsilon}(i)) \ll \sum_{i=i_0(m)}^{\infty} 1/\varphi_{m, \varepsilon}(i) < \infty$. By a classical form of the strong law of large numbers (see e.g.\ \cite[p.\ 209]{P}) and $R^{1/p} \varphi_{m, \varepsilon} (R)^{1/p} = \Theta \left( \varphi_{m,\varepsilon} (N_R)^{1/p} \right)$, we have
\begin{equation}\label{SLLNYi*}
\lim_{R \to \infty}  \frac{\sum_{i=1}^{R} Y_i^*}{\varphi_{m, \varepsilon}(N_R)^{1/p}} =0 \quad \mathrm{a.s.}
\end{equation}
Lemma \ref{approxerrorlemma} gives $\E |Y_i-Y_i^*|^p \ll |J_i|^p \Delta_{|H_i|} \ll i^{p-8}$, and hence $\Pr \left( |Y_i-Y_i^*| \ge 1/i^2 \right) \ll i^{3p-8} \le i^{-2}$. By the Borel--Cantelli lemma $\sum_{i=2}^{\infty} |Y_i-Y_i^*|<\infty$ a.s., and consequently \eqref{SLLNYi*} remains true if we replace $Y_i^*$ by $Y_i$. Repeating the same arguments for $Z_i$ and $Z_i^*$, we obtain \eqref{mainSLLN} along the subsequence $N_R=\max J_R$.

On the other hand, applying Lemma \ref{fluctuationlemma} with $m+2$ on the interval $H_R \cup J_R$ of length $2R$, we get
\[ \Pr \left( \max_{N \in H_R \cup J_R} \left| \sum_{k=\min H_R}^{N} f(S_k) \right| \ge R^{1/p} \varphi_{m+1, \varepsilon} (R)^{1/p} \right) \ll \frac{(\log_{m+2}(R))^p}{\varphi_{m+1,\varepsilon} (R)} . \]
The Borel--Cantelli lemma shows that with probability $1$, for any $R \ge 1$ and any $N \in H_R \cup J_R$ the fluctuation satisfies $\left| \sum_{k=\min H_R}^{N} f(S_k) \right| \ll_{\omega} \varphi_{m+1, \varepsilon}(N)^{1/p}$ with an implied constant depending on the point $\omega$ of the probability space. Therefore \eqref{mainSLLN} holds along all $N$. This finishes the proof of Theorem \ref{maintheoremLIL} (i) under the extra condition that $\nu$ is strictly aperiodic.

\begin{proof}[Proof of Theorem \ref{maintheoremWIENER}] Assume $p=2+\delta$ for some $0<\delta <2$, and $C(f, \nu )>0$. Let us decompose the set of positive integers into consecutive intervals of integers $H_1, J_1, H_2, J_2, \ldots$, as in Section \ref{approxsection} (with the choice $M=0$), such that $|H_i|=\lceil 22 \Delta \log (i+1) \rceil$ and $|J_i|=\lceil i^{\delta /(4+2\delta)} \rceil$ for all $i \ge 1$. As before, we have $\Delta_{|H_i|}\le i^{-11}$ and $\Delta_{|J_i|}\le i^{-11}$ for all integers $i$ large enough in terms of $\nu$ and $\delta$.

Corollary \ref{momentscorollary} (ii) and the Erd\H{o}s--Stechkin inequality \cite[Theorem A]{MO} give $\left\| \max_{1 \le n \le N} \left| \sum_{k=M+1}^{M+n} f(S_k) \right| \right\|_{2+\delta} \ll \sqrt{N}$ for any $M \ge 0$ and $N \ge 1$. Therefore for any $R \ge 1$ we have
\[ \Pr \left( \max_{N \in H_R \cup J_R} \left| \sum_{k=\min H_R}^N f(S_k) \right| \ge R^{1/2} \right) \ll \frac{|J_R|^{1+\delta /2}}{R^{1+\delta /2}} \ll \frac{1}{R^{1+\delta /4}} . \]
The Borel--Cantelli lemma shows that with probability $1$, for any $R \ge 1$ and any $N \in H_R \cup J_R$ the fluctuation satisfies $\left| \sum_{k=\min H_R}^{N} f(S_k) \right| \ll_{\omega} R^{1/2}$ with an implied constant depending on the point $\omega$ of the probability space. For any $t \ge 1$ let $R(t)$ denote the positive integer for which $\lfloor t \rfloor \in H_{R(t)} \cup J_{R(t)}$. Summing over $\min J_1 \le k \le \max J_{R(t)}$ instead of $1 \le k \le t$, we thus obtain
\begin{equation}\label{fSkdecomposition}
\sum_{k \le t} f(S_k) = \sum_{i=1}^{R(t)} \sum_{k \in J_i} f(S_k) + \sum_{i=2}^{R(t)} \sum_{k \in H_i} f(S_k) + O_{\omega} \left( R(t)^{1/2} \right) \quad \mathrm{a.s.}
\end{equation}
Here $\sum_{i=1}^{R(t)} \sum_{k \in J_i} f(S_k) \overset{d}{=} \sum_{i=1}^{R(t)} Y_i$ and $\sum_{i=2}^{R(t)} \sum_{k \in H_i} f(S_k) \overset{d}{=} \sum_{i=2}^{R(t)} Z_i$ in the Skorokhod space $D[0,\infty)$. From Lemma \ref{approxerrorlemma} we get $\E |Y_i-Y_i^*|^{2+\delta} \ll |J_i|^{2+\delta} \Delta_{|H_i|} \ll i^{-11+\delta /2}$. Hence $\Pr (|Y_i-Y_i^*| \ge 1/i^2) \ll i^{-7+5\delta /2} \le i^{-2}$, so by the Borel--Cantelli lemma $\sum_{i=1}^{\infty} |Y_i-Y_i^*| < \infty$ a.s. Clearly the same holds for $Z_i-Z_i^*$.

By Lemma \ref{yi*zi*} we have $\| Z_i^* \|_{2+\delta} \ll \sqrt{\log i}$, and $\sum_{i=2}^R \E |Z_i^*|^2 =\Theta (R \log R)$. It follows (see e.g.\ \cite[p.\ 246]{P}) that $\sum_{i=2}^R Z_i^*$ satisfies the law of the iterated logarithm; in particular, $\left| \sum_{i=2}^R Z_i^* \right| \ll_{\omega} \sqrt{R \log R \log \log R}$ a.s. Note that $R(t)^{1/2} \ll t^{(2+\delta)/(4+3\delta)}$ and $(2+\delta)/(4+3\delta) < 1/2-\delta /20$ whenever $0<\delta <2$. Thus the second double sum on the right hand side of \eqref{fSkdecomposition} is $o(t^{1/2-\delta /20})$ a.s., and consequently the processes $\sum_{k \le t} f(S_k)$ and $\sum_{i=1}^{R(t)} Y_i^*$ are $o\left( t^{1/2-\delta /20} \right)$-equivalent.

A special case of a theorem of Strassen \cite[Theorem 4.4]{STR} states the following. Given independent random variables $\zeta_i$, $i=1,2,\dots$ with $\E \zeta_i =0$ and $V_R=\sum_{i=1}^R \E |\zeta_i|^2 \to \infty$, for any $t \ge V_1$ let $R'(t)$ denote the positive integer for which $V_{R'(t)} \le t < V_{R'(t)+1}$. If $\sum_{i=1}^{\infty} \E |\zeta_i|^p / V_i^{\theta p/2} < \infty$ for some $p>2$ and $0 \le \theta \le 1$, then the processes $\sum_{i=1}^{R'(t)} \zeta_i$ and $W(t)$ are $o\left( t^{(1+\theta )/4} \log t \right)$-equivalent, where $W(t)$ is a standard Wiener process.

We apply Strassen's theorem to $\zeta_i=Y_i^*/\sqrt{C(f, \nu)}$, $i=1,2,\dots$. By Lemma \ref{yi*zi*} we have $V_R=\sum_{i=1}^R \E |\zeta_i|^2 = \sum_{i=1}^R |J_i| + O(R) = \Theta \left( R^{1+\delta / (4+2 \delta)} \right)$ and $\E |\zeta_i|^{2+\delta} \ll |J_i|^{1+\delta /2} \ll i^{\delta /4}$. Hence $\sum_{i=1}^{\infty} \E |\zeta_i|^{2+\delta} / V_i^{\theta (1+\delta /2)} < \infty$ for any $\theta > (4+\delta)/(4+3\delta)$. Choosing $\theta$ close enough to $(4+\delta)/(4+3\delta)$, we have $(1+\theta )/4 < 1/2-\delta /20$, and so the processes $\sum_{i=1}^{R'(t)} Y_i^*/\sqrt{C(f, \nu)}$ and $W(t)$ are $o(t^{1/2-\delta /20})$-equivalent; clearly so are $\sum_{i=1}^{R'(t)}Y_i^*$ and $\sqrt{C(f,\nu )} W(t)$.

Finally, we show that the processes $Y(t)=\sum_{i=1}^{R'(t)}Y_i^*$ and $\sum_{i=1}^{R(t)}Y_i^*$ are $o(t^{1/2 - \delta /20})$-equivalent. Clearly $\max J_R=\sum_{i=1}^R (|H_i|+|J_i|)$, and recall that $V_R = \sum_{i=1}^R |J_i| + O(R)$. Therefore for all large enough integer $r$, on the interval $V_r \le t <V_{r+1}$ we have $R'(t)=r$ and $R(t)=R'(V_r-s)$ for some $0 \le s \ll V_r^{(4+2\delta )/(4+3 \delta)} \log V_r$, and hence $\sum_{i=1}^{R'(t)} Y_i^*=Y(V_r)$ and $\sum_{i=1}^{R(t)} Y_i^*=Y(V_r-s)$. Letting $K_r = c V_r^{(4+2\delta )/(4+3 \delta)} \log V_r$ with a large enough constant $c>0$, it will thus be enough to prove that
\begin{equation}\label{processsup}
\sup_{0 \le s \le K_r} |Y(V_r)-Y(V_r-s)| = o\left( V_r^{1/2-\delta /20} \right) \quad \mathrm{a.s.}
\end{equation}
Recalling the distribution of the running maximum of a Wiener process, we have
\[ \Pr \left( \sup_{0 \le s \le K_r} |W(V_r)-W(V_r-s)| \ge \lambda \sqrt{K_r} \right) \ll e^{-\lambda^2 /2} . \]
Choosing, say, $\lambda = 2 \sqrt{\log V_r}$ and noting $(2+\delta )/(4+3 \delta) < 1/2-\delta /20$, the Borel--Cantelli lemma shows that the process $W(t)$ satisfies the property in \eqref{processsup}; clearly so does $\sqrt{C(f, \nu )} W(t)$. Since \eqref{processsup} is invariant under $o(t^{1/2-\delta /20})$-equivalence, $Y(t)$ also satisfies \eqref{processsup}. This finishes the proof in the case $C(f, \nu )>0$.

If $C(f, \nu )=0$, the proof is much simpler. In this case Lemma \ref{yi*zi*} gives $\E |Y_i^*|^2 \ll 1$. Therefore $\sum_{i=1}^{\infty} \E |Y_i^*|^2/i^{1+2\varepsilon} < \infty$ for any $\varepsilon >0$, and by the strong law of large numbers $\sum_{i=1}^{R(t)} Y_i^* = o\left( R(t)^{1/2+\varepsilon} \right)$ a.s. Similarly, $\sum_{i=2}^{R(t)} Z_i^*=o\left( R(t)^{1/2+\varepsilon} \right)$ a.s. Using these relations instead of the law of the iterated logarithm and Strassen's theorem and noting that $R(t)^{1/2+\varepsilon} \ll t^{1/2-\delta /20}$ for small enough $\varepsilon >0$, we get $\sum_{k \le t} f(S_k)=o\left( t^{1/2-\delta /20} \right)$ a.s., as claimed.
\end{proof}

\subsection{Almost sure asymptotics, general case}\label{section5.2}

\begin{proof}[Proof of Theorem \ref{maintheoremLIL}] Under the extra condition that $\nu$ is strictly aperiodic, we proved claim (i) in Section \ref{section5.1}, whereas claim (ii) follows from Theorem \ref{maintheoremWIENER}. We now show that the condition of strict aperiodicity can be removed, and prove the general case of Theorem \ref{maintheoremLIL}.

Assume that the pair $(G,\nu)$ satisfies the conditions of Theorem \ref{maintheoremLIL}; that is, $G$ is a compact metrizable group, and $\nu$ is a Borel probability measure on $G$ such that $\nu$ is adapted, and $(\nu^{*k})_{\mathrm{abs}} \neq 0$ for some $k \ge 1$. We shall use the notation $\mu_G$ for the normalized Haar measure on $G$. It is not difficult to see that if $\nu_1$ and $\nu_2$ are Borel probability measures on $G$, then $\mathrm{supp}\, (\nu_1 * \nu_2)=(\mathrm{supp}\, \nu_1)(\mathrm{supp}\, \nu_2)$ (see e.g.\ \cite[Lemma 2]{U}). Therefore $\mathrm{supp}\, \nu^{*k}=(\mathrm{supp}\, \nu)^k$, where we use the notation $A^k=\{ a_1 a_2 \cdots a_k : a_1, a_2, \dots, a_k \in A \}$. In particular, $\mathrm{supp}\, \nu^{*(k+1)}$ contains a translate of $\mathrm{supp}\, \nu^{*k}$, so the sequence $\mu_G (\mathrm{supp}\, \nu^{*k})$ is nondecreasing. Let $\alpha (G,\nu) = \lim_{k \to \infty} \mu_G (\mathrm{supp}\, \nu^{*k})$. Note that $\alpha (G, \nu) >0$ because $(\nu^{*k})_{\mathrm{abs}} \neq 0$ for some $k \ge 1$. The following simple observation is a special case of \cite[Theorem 14]{G}. For the sake of completeness we include a short proof.
\begin{lem}\label{subgrouplemma} Let $G$ be a compact metrizable group. If $K \subseteq G$ is nonempty and closed, and $K^2 \subseteq K$, then $K$ is a subgroup.
\end{lem}

\begin{proof} Let $a \in K$ be arbitrary. By assumption $a^n \in K$ for all $n \ge 1$. Using the compactness of $K$ we have $a^{n_k} \to b \in K$ as $k \to \infty$ for some subsequence $a^{n_k}$. For any fixed $n \ge 1$ we have $a^{n_k-n} \to a^{-n}b \in K$ as $k \to \infty$. After replacing $n_k$ by another subsequence we may assume that $a^{n_k} \to b \in K$ and $a^{-n_k}b \to c \in K$ as $k \to \infty$ for some $c$. Then $b=a^{n_k}a^{-n_k}b \to bc$, hence $c=1 \in K$. It remains to prove that for any $a \in K$ we have $a^{-1} \in K$. But $aK$ is also nonempty and closed, and $(aK)^2 \subseteq aK$. By the previous argument $1 \in aK$, therefore $a^{-1} \in K$.
\end{proof}

Assume now, that there exists a proper closed normal subgroup $H\lhd G$ such that $\mathrm{supp}\, \nu \subseteq aH$ for some coset $aH$. Since $H$ is normal, we have $\mathrm{supp}\, \nu^{*k} \subseteq a^kH$ for all $k \ge 1$. Thus $\mu_G (H)=\mu_G (a^kH) \ge \mu_G (\mathrm{supp}\, \nu^{*k})$, and so $\mu_G (H) \ge \alpha (G, \nu) >0$. In particular, $aH$ has finite order $d$ in the factor group $G/H$. Since $\bigcup_{i=1}^{d}a^i H$ is a closed subgroup of $G$ containing $\mathrm{supp}\, \nu$, and $\nu$ is assumed to be adapted, we have $G=\bigcup_{i=1}^{d}a^i H$ and $|G:H|=d$. As $\mathrm{supp}\, \nu^{*d} \subseteq H$, we can view $\nu^{*d}$ as a Borel probability measure on the compact metrizable group $H$. Note that $\mu_H (B)=d \cdot \mu_G (B)$ ($B \subseteq H$ Borel) is the normalized Haar measure on $H$. Clearly $(\nu^{*d})^{*k}$ has an absolutely continuous component with respect to $\mu_H$ for some $k \ge 1$. It is also not difficult to see that $\nu^{*d}$ is adapted on $H$. Indeed, suppose $K<H$ is a proper closed subgroup for which $\mathrm{supp}\, \nu^{*d} \subseteq K$. Consider $C=\bigcup_{i=1}^{d} \mathrm{supp}\, \nu^{*i}K$, and note that here $\mathrm{supp}\, \nu^{*i}K \subseteq a^iH$; in particular, $C \neq G$. On the other hand, writing an arbitrary integer $k \ge 1$ in the form $k=nd+i$, $1 \le i \le d$ we have $\mathrm{supp}\, \nu^{*k} = (\mathrm{supp}\, \nu^{*i}) (\mathrm{supp}\, \nu^{*d})^n \subseteq \mathrm{supp}\, \nu^{*i} K$. Therefore the topological closure $\overline{\bigcup_{k=1}^{\infty} \mathrm{supp}\, \nu^{*k}}$ is a subset of $C \neq G$. Using Lemma \ref{subgrouplemma} we get that $\overline{\bigcup_{k=1}^{\infty} \mathrm{supp}\, \nu^{*k}}$ is a proper closed subgroup of $G$, contradicting the adaptedness of $\nu$. Altogether, we find that the pair $(H,\nu^{*d})$ satisfies the conditions of Theorem \ref{maintheoremLIL}. Observe, moreover, that $\alpha (H, \nu^{*d}) = d \cdot \alpha (G, \nu)$.

Assume in addition, that the pair $(H, \nu^{*d})$ satisfies the claims of Theorem \ref{maintheoremLIL}. We now prove that under all these assumptions $(G, \nu)$ also satisfies the claims of Theorem \ref{maintheoremLIL}. Fix a Borel measurable function $f:G \to \mathbb{R}$ such that $\sup_{c \in G} |f(cX_1)|^p < \infty$ for some $p \ge 1$. It will be enough to prove that for any $1 \le i \le d$ we have
\begin{equation}\label{sumfi1}
\lim_{N \to \infty} \frac{\left| \sum_{k=1}^N f(S_{i+kd}) - d N \int_{a^i H} f \, \mathrm{d} \mu_G \right|}{\varphi_{m,\varepsilon}(N)^{1/p}} =0 \qquad \mathrm{a.s.}
\end{equation}
for any $m \ge 1$ and $\varepsilon >0$ in the case $1 \le p \le 2$, and
\begin{equation}\label{sumfi2}
\limsup_{N \to \infty} \frac{\left| \sum_{k=1}^N f(S_{i+kd}) - d N \int_{a^i H} f \, \mathrm{d} \mu_G \right|}{\sqrt{N \log \log N}} < \infty \qquad \mathrm{a.s.}
\end{equation}
in the case $p>2$. Fix $1 \le i \le d$, and let $\mathcal{F}_i$ denote the $\sigma$-algebra generated by $X_1, X_2, \dots , X_i$. Letting $Y_n=\prod_{j=i+(n-1)d+1}^{i+nd} X_j$, the variables $Y_1, Y_2, \dots$ are i.i.d.\ $H$-valued random variables with distribution $\nu^{*d}$, independent of $\mathcal{F}_i$. Let $b=X_1 X_2 \cdots X_i$, and note $b \in a^i H$ a.s. Let $g:H \to \mathbb{R}$, $g(x)=f(bx)$, and observe $\sup_{c \in H} \E \left( |g(cY_1)|^p \mid \mathcal{F}_i \right) < \infty$ a.s.\ and $\int_H g \, \mathrm{d} \mu_H = d \int_{a^i H} f \, \mathrm{d}\mu_G$ a.s. We thus have
\[ \sum_{k=1}^N f(S_{i+kd}) - d N \int_{a^i H} f \, \mathrm{d} \mu_G = \sum_{k=1}^N g(Y_1 Y_2 \cdots Y_k) -N \int_H g \, \mathrm{d}\mu_H \quad \mathrm{a.s.} \]
By the assumption that $(H,\nu^{*d})$ satisfies the claims of Theorem \ref{maintheoremLIL}, we have
\[ \Pr \left( \lim_{N \to \infty} \frac{\left| \sum_{k=1}^N g(Y_1 Y_2 \cdots Y_k)-N \int_H g \, \mathrm{d}\mu_H \right|}{\varphi_{m, \varepsilon}(N)^{1/p}} =0 \,\,\,\, \vline \,\,\,\, \mathcal{F}_i \right) =1 \]
in the case $1 \le p \le 2$, and
\[ \Pr \left( \limsup_{N \to \infty} \frac{\left| \sum_{k=1}^N g(Y_1 Y_2 \cdots Y_k)-N \int_H g \, \mathrm{d}\mu_H \right|}{\sqrt{N \log \log N}} < \infty \,\,\,\, \vline \,\,\,\, \mathcal{F}_i \right) =1. \]
in the case $p>2$. Taking the (total) probability, \eqref{sumfi1} and \eqref{sumfi2} follow.

Finally, we prove that the pair $(G,\nu)$ satisfies the claims of Theorem \ref{maintheoremLIL}. Let $H_0=G$. If $\nu$ is not strictly aperiodic in $H_0$, then let $H_1 \lhd H_0$ be a proper closed normal subgroup such that $\mathrm{supp} \, \nu$ is contained in a coset of $H_1$, and let $d_1=$ $|H_0:H_1|$. As seen above, the pair $(H_1, \nu^{*d_1})$ satisfies the conditions of Theorem \ref{maintheoremLIL}, hence we can iterate this procedure. We obtain a sequence $H_0 \rhd H_1 \rhd \cdots \rhd H_j$, where $H_i$ is a proper closed normal subgroup of $H_{i-1}$ with index $d_i=|H_{i-1}:H_i|$, and $\mathrm{supp} \, \nu^{*(d_1\cdots d_{i-1})}$ is contained in a coset of $H_i$ for all $1 \le i \le j$. The procedure ends after step $j$ if $\nu^{*(d_1 \cdots d_j)}$ is strictly aperiodic in $H_j$. Note that $1 \ge \alpha (H_i, \nu^{*(d_1 \cdots d_i)}) = d_1 \cdots d_i \alpha (G,\nu)$, therefore the procedure terminates after finitely many steps. We prove the claims by induction on $j$. If $j=0$, that is, $\nu$ is strictly aperiodic, the claims have already been proved. To prove the inductive step from $j-1$ to $j$, we first apply the inductive hypothesis to $(H_1,\nu^{*d_1})$, then the arguments above to conclude that $(G,\nu)$ satisfies the claims of Theorem \ref{maintheoremLIL}.
\end{proof}

\begin{proof}[Proof of Theorem \ref{stronguniformityLIL}] The implication (iii)$\Rightarrow$(ii) is trivial, whereas (i)$\Rightarrow$(iii) is a special case of Theorem \ref{maintheoremLIL}. Let us finally prove (ii)$\Rightarrow$(i). First, suppose that $\nu^{*k}$ is singular with respect to $\mu$ for every $k \ge 1$. Then there exists a Borel set $B \subseteq G$ such that $\mu (B)=0$ and $\Pr (S_k \in B)=1$ for every $k \ge 1$. Hence the indicator function $f=I_B$ does not satisfy (ii), giving a contradiction. Suppose next, that $\nu$ is not adapted; that is, there exists a proper closed subgroup $H<G$ such that $\Pr (X_1 \in H)=1$. Then $\Pr (S_k \in H)=1$ for all $k \ge 1$. Since every nonempty open subset of $G$ has positive Haar measure, we have $\mu (H)<1$. Therefore $f=I_H$ does not satisfy (ii), giving a contradiction.
\end{proof}

\subsection{Central limit theorem}\label{section5.3}

\begin{proof}[Proof of Theorem \ref{maintheoremCLT}] In this proof implied constants will be universal. Claim (ii) follows from Corollary \ref{momentscorollary} (i). To see (i), fix a positive integer $N$ large enough in terms of $f$, $\nu$ and $\delta$, and let us prove \eqref{BerryEsseen}. Let $E_N=N^{-\delta/(2+2\delta )} \log^{\delta/(1+\delta)} N$ and $K=\Delta \left( \| f \|_{2+\delta} /\sqrt{C(f, \nu)} \right)^{(2+\delta)/(1+\delta)}$. Let us decompose the set $\{ 1,2,\dots, N \}$ into consecutive intervals of integers $H_1, J_1, \dots, H_R, J_R$, as in Section \ref{approxsection} (with the choice $M=0$), such that $|H_i|= \lceil 4 \Delta \log N \rceil$ and $|J_i|= \Theta \left( (\Delta /K^2) N ^{\delta/(1+\delta)} \log^{2/(1+\delta)}N \right)$ for all $1 \le i \le R$. As in the proof of Corollary \ref{momentscorollary}, we have $\Delta_{|H_i|} \le N^{-2}$, and clearly the same holds for $\Delta_{|J_i|}$.

Recall that
\[ \sum_{k=1}^N f(S_k) = \sum_{k \in H_1} f(S_k) + \sum_{i=1}^R \sum_{k \in J_i} f(S_k) + \sum_{i=2}^R \sum_{k \in H_i} f(S_k), \]
where $\sum_{i=1}^R \sum_{k \in J_i} f(S_k) \overset{d}{=} \sum_{i=1}^R Y_i$ and $\sum_{i=2}^R \sum_{k \in H_i} f(S_k) \overset{d}{=} \sum_{i=2}^R Z_i$. From Lemma \ref{yi*zi*} and the classical Lyapunov condition (see e.g.\ \cite[p.\ 154]{P}) we get
\begin{equation}\label{yi*berryesseen}
\Pr \left( \frac{\sum_{i=1}^R Y_i^*}{\sqrt{\sum_{i=1}^R \E (Y_i^*)^2}} <x \right) = \Phi (x) + O \left( K E_N \right) .
\end{equation}
Here $\sum_{i=1}^R \E (Y_i^*)^2 = C(f,\nu ) N+O_{f,\nu}(N^{1/(1+\delta)})$, therefore the error of replacing the normalizing factor on the left hand side of \eqref{yi*berryesseen} by $\sqrt{C(f,\nu ) N}$ is $o(E_N)$. Similarly, $\sum_{i=2}^R Z_i^*$ also satisfies the central limit theorem with remainder term $O(K E_N)$. In particular,
\[ \sup_{x \in \mathbb{R}} \left| \Pr \left( \frac{\left|\sum_{i=2}^R Z_i^*\right|}{\sqrt{\sum_{i=2}^R \E (Z_i^*)^2}} \ge x \right) - 2(1-\Phi (x)) \right| \ll K E_N . \]
Applying this with $x=\sqrt{\log N}$ and noting that $1-\Phi(\sqrt{\log N})=O(N^{-1/2})=o(E_N)$, we obtain
\begin{equation}\label{zi*upperbound}
\Pr \left( \frac{\left| \sum_{i=2}^R Z_i^* \right|}{\sqrt{C(f,\nu )N}} \gg K E_N \right) \ll K E_N.
\end{equation}
From Lemma \ref{approxerrorlemma} we get $\| \sum_{i=1}^R (Y_i-Y_i^*) \|_2 \ll_{f,\nu} \sum_{i=1}^R |J_i| \sqrt{\Delta_{|H_i|}} \ll_{f,\nu} 1$, hence the Chebyshev inequality gives
\begin{equation}\label{yi-yi*upperbound}
\Pr \left( \frac{\left| \sum_{i=1}^R (Y_i-Y_i^*) \right|}{\sqrt{C(f,\nu )N}} \ge K E_N \right) \ll_{f,\nu} \frac{1}{N E_N^2} = o(E_N).
\end{equation}
We similarly deduce
\begin{equation}\label{zi-zi*upperbound}
\Pr \left( \frac{\left| \sum_{i=2}^R (Z_i-Z_i^*) \right|}{\sqrt{C(f,\nu )N}} \ge K E_N \right) \ll_{f,\nu} \frac{1}{N E_N^2} = o(E_N).
\end{equation}
Finally, note that $\sup_{c \in G} \E f(cX_1)^2 < \infty$ implies $\sup_{k \ge 1} \E f(S_k)^2 < \infty$. Therefore $\| \sum_{k \in H_1} f(S_k) \|_2 \ll_{f,\nu} \log N$, and the Chebyshev inequality gives
\begin{equation}\label{sumH1}
\Pr \left( \frac{\left|\sum_{k \in H_1} f(S_k) \right|}{\sqrt{C(f,\nu ) N}} \ge K E_N \right) \ll_{f,\nu} \frac{\log N}{N E_N^2} = o(E_N).
\end{equation}
Combining \eqref{yi*berryesseen}--\eqref{sumH1} we thus have
\[ \begin{split} \Pr \left( \frac{\sum_{k=1}^N f(S_k)}{\sqrt{C(f, \nu )N}} <x \right) &= \Pr \left( \frac{\sum_{i=1}^R Y_i}{\sqrt{C(f, \nu )N}} <x+O(K E_N) \right) +O(K E_N) \\ &= \Phi (x) + O(K E_N). \end{split} \]
\end{proof}

We now prove the Remark made after Theorem \ref{maintheoremCLT}. If $f \in L^4(G)$, then instead of $\| Y_i^* \|_3 \ll \| f \|_3 \sqrt{\Delta |J_i|}$, Lemma \ref{yi*zi*} gives the slightly better estimate $\| Y_i^* \|_3 \le \| Y_i^* \|_4 \ll \| f \|_2 \sqrt{\Delta |J_i|} + \| f \|_4 \Delta^{3/4} |J_i|^{1/4}$. Therefore if in the definition of $K$ we replace $\| f \|_{2+\delta}$ by $\| f \|_2$, the Lyapunov condition gives that \eqref{yi*berryesseen} and \eqref{zi*upperbound} hold with error terms $O(KE_N)+o(E_N)=O(K E_N)$. The rest of the proof remains unchanged.

\begin{proof}[Proof of Theorem \ref{stronguniformityCLT}] The implication (i)$\Rightarrow$(ii) follows from Theorem \ref{maintheoremCLT} and Proposition \ref{specialcasesprop}. The latter is needed to ensure $C(f,\nu) >0$. We now prove (ii)$\Rightarrow$(i). The facts that $\nu$ is adapted, and that $(\nu^{*k})_{\mathrm{abs}} \neq 0$ for some $k \ge 1$ follow similarly to the proof of Theorem \ref{stronguniformityLIL}. Suppose that $\mathrm{supp}\, \nu$ is contained in a coset $aH$ of some proper closed normal subgroup $H \lhd G$. We have seen in Section \ref{section5.2} that the index $d=|G:H|$ is finite, and $G=\bigcup_{i=1}^d a^i H$. Note that if $k=nd+i$ for some $1 \le i \le d$, then $S_k \in a^i H$ a.s. Letting $f=I_H-\mu (H)$, we thus have $\left| \sum_{k=1}^N f(S_k) \right| \le 1-1/d$ a.s. Hence $N^{-1/2}\sum_{k=1}^{N} f(S_k)$ cannot have a nondegenerate limit distribution.
\end{proof}


\begin{thebibliography}{99}
\footnotesize{

\bibitem{AG} M.\ Anoussis and D.\ Gatzouras: \textit{A spectral radius formula for the Fourier transform on compact groups and applications to random walks.} Adv.\ Math.\ 188 (2004), no.\ 2, 425--443.

\bibitem{BE} A.\ Berger and S.\ N.\ Evans: \textit{A limit theorem for occupation measures of L\'evy processes in compact groups.} Stoch.\ Dyn.\ 13 (2013), no.\ 1, 1250008, 16 pp.

\bibitem{BB} I.\ Berkes and B.\ Borda: \textit{On the law of the iterated logarithm for random exponential sums.} Trans.\ Amer.\ Math.\ Soc.\ 371 (2019), no.\ 5, 3259--3280.

\bibitem{BR} I.\ Berkes and M.\ Raseta: \textit{On the discrepancy and empirical distribution function of $\{ n_k \alpha \}$.} Unif.\ Distrib.\ Theory 10 (2015), no.\ 1, 1--17.

\bibitem{B} R.\ N.\ Bhattacharya: \textit{Speed of convergence of the $n$-fold convolution of a probability measure on a compact group.} Z.\ Wahrscheinlichkeitstheorie und Verw.\ Gebiete 25 (1972/73), 1--10.

\bibitem{FO} G.\ B.\ Folland: \textit{A Course in Abstract Harmonic Analysis.} Second edition. CRC Press, Boca Raton, FL, 2016.

\bibitem{G} B.\ Gelbaum, G.\ K.\ Kalisch and J.\ M.\ H.\ Olmsted: \textit{On the embedding of topological semigroups and integral domains.} Proc.\ Amer.\ Math.\ Soc.\ 2 (1951), 807--821.

\bibitem{KI} Y.\ Kawada and K.\ It\^o: \textit{On the probability distribution on a compact group. I.} Proc.\ Phys.-Math.\ Soc.\ Japan (3) 22 (1940), 977--998.

\bibitem{K} B.\ M.\ Kloss: \textit{Probability distributions on bicompact topological groups.} Theor.\ Probability Appl.\ 4 (1959), 237--270.

\bibitem{L} P.\ L\'evy: \textit{L'addition des variables al\'eatoires d\'efinies sur un circonf\'erence.} Bull.\ Soc.\ Math.\ France 67 (1939), 1--41.

\bibitem{MO} F.\ M\'oricz: \textit{Moment inequalities and the strong laws of large numbers.} Z.\ Wahrscheinlichkeitstheorie und Verw.\ Gebiete 35 (1976), no.\ 4, 299--314.

\bibitem{P} V.\ V.\ Petrov: \textit{Limit Theorems of Probability Theory. Sequences of Independent Random Variables.} Oxford Studies in Probability, 4. Oxford Science Publications. The Clarendon Press, Oxford University Press, New York, 1995.

\bibitem{RX} K.\ A.\ Ross and D.\ Xu: \textit{Norm convergence of random walks on compact hypergroups.} Math.\ Z.\ 214 (1993), no.\ 3, 415--423.

\bibitem{SCH} P.\ Schatte: \textit{On a law of the iterated logarithm for sums mod 1 with application to Benford's law.} Probab.\ Theory Related Fields 77 (1988), no.\ 2, 167--178.

\bibitem{STR} V.\ Strassen: \textit{Almost sure behavior of sums of independent random variables and martingales.} 1967 Proc.\ Fifth Berkeley Sympos.\ Math.\ Statist.\ and Probability (Berkeley, Calif., 1965/66) Vol.\ II: Contributions to Probability Theory, Part 1, pp.\ 315--343 Univ.\ California Press, Berkeley, Calif.

\bibitem{ST} K.\ Stromberg: \textit{Probabilities on a compact group.} Trans.\ Amer.\ Math.\ Soc.\ 94 (1960), 295--309.

\bibitem{U} K.\ Urbanik: \textit{On the limiting probability distribution on a compact topological group.} Fund.\ Math.\ 44 (1957), 253--261.

\bibitem{V} C.\ Villani: \textit{Topics in Optimal Transportation.} Graduate Studies in Mathematics, 58. American Mathematical Society, Providence, RI, 2003.
}
\end{thebibliography}
\end{document}